\documentclass[12pt]{amsart}

\usepackage{amsmath,amssymb,amsfonts,fullpage,enumerate,graphicx,caption,subcaption}

\newtheorem{thm}{Theorem}[section]
\newtheorem{theorem}[thm]{Theorem}
\newtheorem{cor}[thm]{Corollary}
\newtheorem{proposition}[thm]{Proposition}
\newtheorem{obs}[thm]{Observation}
\newtheorem{lemma}[thm]{Lemma}
\newtheorem{conj}[thm]{Conjecture}
\newtheorem*{problem}{Problem}

\theoremstyle{definition}
\newtheorem{defn}[thm]{Definition}

\makeatletter
\let\c@equation\c@thm
\makeatother
\numberwithin{equation}{section}

\newcommand{\cart}{\, \Box \,}
\newcommand{\ceil}[1]{\left \lceil #1 \right \rceil}
\newcommand{\floor}[1]{\left \lfloor #1 \right \rfloor}

\DeclareMathOperator{\rad}{rad}
\DeclareMathOperator{\capture}{c}
\newcommand{\search}{\capture'}
\newcommand{\lcap}{\capture_{\ell}}
\newcommand{\lloc}{\search_{\ell}}
\DeclareMathOperator{\weakmonocap}{mc}
\newcommand{\weakmonoloc}{\weakmonocap'}

\newcommand{\lweakmonocap}{\weakmonocap_{\ell}}
\newcommand{\lweakmonoloc}{\weakmonoloc_{\ell}}

\begin{document}

\bibliographystyle{plain}

\title{Limited-visibility Cops and Robbers on Hamming Graphs}

\author{John Jones}
\address{Department of Mathematics and Applied Mathematical Sciences, University of Rhode Island, University of Rhode Island, Kingston, RI, USA, 02881}
\email{\tt jj\_jones@uri.edu}

\author{William B. Kinnersley}
\address{Department of Mathematics and Applied Mathematical Sciences, University of Rhode Island, University of Rhode Island, Kingston, RI, USA, 02881}
\email{\tt billk@uri.edu}


\begin{abstract}
In the classic game of \textit{Cops and Robbers}, a team of cops pursues a robber through a graph.  The traditional model of Cops and Robbers operates under the assumption that the cops know the robber's location at all times.  Recently, however, there has been growing interest in models wherein the cops have only partial information about the robber's location.  In this paper, we study the \textit{limited-visibility} variant of Cops and Robbers, in which the cops only know the robber's position if some cop is within a fixed distance of the robber.  For this variant of the game, we give bounds on the number of cops needed to capture a robber on Hamming graphs, and we use these results to settle several open problems posed by Clarke et al. \cite{CCDDFM20}.
\end{abstract}

\maketitle

\begin{section}{Introduction}

\textit{Cops and Robbers} is a well-studied model of pursuit and evasion, in which a team of cops chases a robber through a graph.  The cops and robber all occupy vertices of the graph, and they take turns moving from vertex to adjacent vertex.  If any cop ever occupies the same vertex as the robber, then the cops \textit{capture} the robber and win the game; the robber wins if he can evade the cops indefinitely.  In the traditional model of Cops and Robbers, the players have perfect information -- the cops always know the location of the robber, and vice-versa.  

Cops and Robbers was introduced independently by Quillot \cite{Qui78} and by Nowakowski and Winkler \cite{NW83}; both papers characterized the graphs on which a single cop can capture the robber.  Subsequently, Aigner and Fromme \cite{AF84} considered the game with more than one cop.  Since then, the main focus of study in the area has been to determine the minimum number of cops needed to capture a robber on a given graph $G$; this quantity is known as the \textit{cop number} of $G$ and is denoted $c(G)$.

Cops and Robbers has been studied extensively, and many variants have been introduced; see \cite{BN11} for a thorough survey.  Recently, there has been increasing interest in models of the game wherein the players have incomplete information.  For example, one can consider the model wherein the robber is ``invisible'' to the cops, meaning that the cops do not know the robber's location until they achieve capture.  The minimum number of cops needed to guarantee capture of a robber on a graph $G$ under this model is the \textit{zero-visibility cop number} of $G$, denoted $c_0(G)$.  This parameter was introduced by Dereniowski, Dyer, Tifenbach, and Yang \cite{DDTY13, DDTY15a, DDTY15b}, who showed that $c_0(G)$ is bounded above by the pathwidth of $G$, characterized those trees $T$ satisfying $c_0(T) = k$ for given $k$, gave a linear-time algorithm for computing the zero-visibility cop number of a tree, and showed that the problem of computing $c_0(G)$ is NP-hard in general.

This model of Cops and Robbers was generalized by Clarke, Cox, Duffy, Dyer, Fitzpatrick, and Messinger \cite{CCDDFM20}, who introduced a variant of the game wherein the cops can only ``see'' vertices that are ``close'' to a cop.  More precisely, in the game of \textit{$\ell$-visibility Cops and Robbers} for a nonnegative integer $\ell$, the cops do not know the robber's location unless some cop occupies a vertex within distance $\ell$ of the robber.  (The robber still knows the cops' locations at all times.)  Clarke et al. introduced two parameters related to this game: $\lloc(G)$ denotes the minimum number of cops needed to ensure that the cops can eventually see the robber when playing the game on the graph $G$, while $\lcap(G)$ denotes the minimum number of cops needed to actually capture the robber.  They established several relationships between $\lloc(G)$, $\lcap(G)$, and $c(G)$, and they gave a variety of results regarding the behavior of $\lcap$ on trees. %*** SAY SOMETHING ABOUT \cite{Tan04} -- 1-visibility? ***

% \comment{We might want to add a word about possible applications of the limited-visibility model.}

In this paper, we focus on determining the $\ell$-visibility cop numbers of \textit{Hamming graphs}, that is, Cartesian products of complete graphs.  The paper is laid out as follows.  In Section \ref{sec:definitions}, we give a formal definition of the $\ell$-visibility game and put forth some preliminary observations that will assist in our analysis of the game.  In Section \ref{sec:dimension_2}, we study the limited-visibility game on two-dimensional Hamming graphs: in Theorems 3.2 and 3.3 we determine $\search_1(K_n \cart K_n)$ and $\capture_1(K_n \cart K_n)$, and in Theorem 3.4 we determine $\search_0(K_n \cart K_n)$ and $\capture_0(K_n \cart K_n)$.  In Section \ref{sec:high_dimension}, we consider the $(d-1)$-visibility game on $d$-dimensional Hamming graphs; in particular, we show in Theorem \ref{thm:Gnd_lower} that $\search_{d-1}(H(d,n)) \ge n/(d+1)$, where $H(d,n)$ denotes the $d$-fold Cartesian product of $K_n$, and we give general bounds on $\search_{\ell}(H(d,n))$.  Finally, in Section \ref{sec:conclusions}, we use these results to resolve Problems 5.3 and 5.5 from \cite{CCDDFM20}, as well as to give a partial resolution to Problem 5.1.\\
\end{section}

\begin{section}{Definitions and Notation}\label{sec:definitions}

\begin{subsection}{Cops and Robbers}

%In this section we give a formal definition of Cops and Robbers and of the limited-visibility variant, and we provide an alternative characterization of the latter that is often simpler to work with.  Most of the notation and definitions we introduce in this section follow Dereniowski et al. \cite{DDTY13} and of Clarke et al. \cite{CCDDFM20}.

The game of \textit{Cops and Robbers} has two players: a team of one or more \textit{cops} and a single \textit{robber}.  (Throughout this paper, for convenience in distinguishing between the two players, we sometimes use the pronouns she/her/hers to refer to a cop and the pronouns he/him/his to refer to the robber.)  All cops and the robber occupy vertices of some fixed graph $G$; multiple cops may simultaneously occupy a single vertex.  At the beginning of the game, the cops choose their initial positions, after which the robber views the cops' positions and chooses his position in response.  Henceforth, the game proceeds in \textit{rounds}, each of which consists of a \textit{cop turn} followed by a \textit{robber turn}.  On the cop turn, each cop either remains on their current vertex or moves to an adjacent vertex.  Likewise, on the robber's turn, the robber may either remain in place or move to an adjacent vertex.  If some cop ever occupies the same vertex as the robber, then that cop \textit{captures} the robber and the cops win the game.  Conversely, the robber wins if he is able to evade capture indefinitely, no matter how the cops play.  The usual focus when studying Cops and Robbers is to determine, for a given graph $G$, the minimum number of cops needed to ensure capture of a robber on $G$; this quantity is the \textit{cop number} of $G$, denoted $c(G)$.  

In the original model of Cops and Robbers, the players have perfect information: at all times, the cops know the location of the robber and vice-versa.  This assumption of perfect information, while convenient, is not always reasonable.  In this paper, we consider a variant of the game in which the cops can only ``see'' the robber when he is ``close to'' some cop.  For a nonnegative integer $\ell$, the game of \textit{$\ell$-visibility Cops and Robbers} is played similarly to the original game, with one change: the robber's location is not revealed to the cops unless some cop occupies a vertex within distance $\ell$ of the robber.  (When this happens, we say that the cop \textit{sees} the robber, and the robber's position becomes known to all of the cops -- even those that do not themselves see the robber.)  The robber still knows the cops' locations at all times.  As in the original game, the cops aim to capture the robber by occupying the robber's vertex.  However, we now require that the cops be able to guarantee capture of the robber within a finite number of rounds; this prevents the cops from moving randomly and ``luckily'' playing an optimal strategy for the perfect-information game.  Equivalently, we may suppose that the robber is ``omniscient'', meaning that he is able to anticipate the cops' movements and respond optimally.  (A variant of limited-visibility cops and robbers without this restriction is considered in \cite{KMP13}.)  $\ell$-visibility Cops and Robbers was first introduced in full generality by Clarke et al. \cite{CCDDFM20}, although the special case $\ell=0$ had previously been considered by Dereniowski, Dyer, Tifenbach, and Yang \cite{DDTY13, DDTY15a, DDTY15b}, and the case $\ell=1$ had been considered by Tang \cite{Tan04}.

In $\ell$-visibility Cops and Robbers, the game often consists of two distinct ``phases''.  
%In the first phase, the cops simply aim to determine the robber's position; we refer to this as the \textit{search phase}.  In the second phase of the game the cops, having located the robber, seek to capture him; we refer to this as the \textit{capture phase}.  
In the first phase, the cops simply aim to determine the robber's position; in the second phase the cops, having located the robber, seek to capture him.  Thus motivated, Clarke et al. introduced two graph parameters relating to this game.  The minimum number of cops needed to ensure capture of a robber on $G$ is the \textit{$\ell$-visibility cop number} of $G$, denoted $\lcap(G)$, while the minimum number of cops needed to see -- but not necessarily capture -- the robber is denoted by $\lloc(G)$.  Clearly $\lloc(G) \le \lcap(G)$, since the cops cannot capture the robber without also seeing him.  Clarke et al. \cite{CCDDFM20} also established the following connection between $\lloc(G), \lcap(G)$, and $c(G)$:

\begin{theorem}[\cite{CCDDFM20}, Theorem 3.1]\label{thm:seeing_plus_one}
For any graph $G$ and any $\ell \ge 2$, we have
$$\lcap(G) \le \max\{\lloc(G), c(G)+1\}.$$
\end{theorem}

As Theorem \ref{thm:seeing_plus_one} suggests, the search phase of the game is often the bottleneck for the cops.  The cops typically cannot hope to locate the robber in a single turn; rather, they must gradually whittle down the set of possible locations for the robber over the course of several rounds.  Following \cite{CCDDFM20}, we introduce the following terminology.  At any point in the game, we say that a vertex is \textit{clean} if it is seen by the cops, or if it is unseen but the robber is known not to be located there; any vertex that is not clean is said to be \textit{dirty}.  
%(For technical reasons, we consider the vertex that actually contains the robber to be clean if the cops see it, and dirty otherwise.)  
A vertex that was once clean, but subsequently becomes dirty, is said to have been \textit{recontaminated}.    

Note that the cops cannot clean all vertices of the graph without seeing the robber, and conversely, the cops cannot guarantee seeing the robber without cleaning all vertices; thus, proving that $\lloc(G) \le k$ is equivalent to proving that $k$ cops can clean all vertices of $G$.  (When establishing bounds on $\lloc(G)$, we ignore the prospect of capturing the robber, and say that the cops win if they ever see the robber; thus, the robber aims solely to evade detection.)  Dirty vertices can only become clean during a cop turn, by virtue of cops moving into new locations; similarly, clean vertices can only become dirty during a robber turn.  In fact, during a robber turn, every clean vertex that is both adjacent to a dirty vertex and not within distance $\ell$ of any cop will necessarily become recontaminated.  Thus, we sometimes refer to the robber turn in a given round of the game as the \textit{recontamination phase} of that round.

We say that a cop strategy is \textit{(strictly) monotonic} if no recontamination ever occurs; that is, once a vertex is clean, it remains so for the remainder of the game.  Dyer et al. \cite{DDTY15b} showed that zero-visibility Cops and Robbers is not monotonic, i.e. there may exist non-monotonic cop strategies that use fewer cops than any monotonic strategy.  This is due in part to the usefulness of a technique known as \textit{vibrating}, wherein a cop moves from a vertex $u$ to a vertex $v$ and then, on the next turn, moves back to $u$.  When a cop vibrates, one or more vertices may briefly become recontaminated but will be cleaned before the next robber turn.  Thus, even though recontamination occurs, the recontaminated vertices are never given a chance to recontaminate their neighbors.  Thus motivated, Dyer et al. defined a cop strategy to be \textit{weakly monotonic} if any vertices that become recontaminated are subsequently cleaned on the following cop turn.

Formally, let $\mathcal{R}_i$ denote the set of dirty vertices immediately prior to the $i^{th}$ cop turn and let $\mathcal{S}_i$ denote denote the set of dirty vertices immediately after the $i^{th}$ cop turn.  We say that a cop strategy is weakly monotonic if $\mathcal{S}_i \subseteq \mathcal{S}_{i+1}$ for all $i$; we say that the strategy is (strictly) monotonic if $\mathcal{S}_i \subseteq \mathcal{S}_{i+1}$ and $\mathcal{R}_i \subseteq \mathcal{R}_{i+1}$ for all $i$.  % This is redundant...

%\comment{*** We might want to eliminate the strictly monotonic parameters, at least for a journal submission (the strictly monotonic stuff could still go in the dissertation ***}

As it turns out, $\ell$-visibility Cops and Robbers is also not weakly monotonic; that is, there may be non-weakly-monotonic cop strategies that use fewer cops than any weakly monotonic strategy (as shown by Dyer et al. \cite{DDTY15b} for the case $\ell=0$ and by Clarke et al. \cite{CCDDFM20} for $\ell \ge 1$).  However, weak monotonicity is still a desirable feature of a cop strategy, and as such one may be interested in determining optimal weakly monotonic strategies.  Thus motivated, let $\lweakmonoloc(G)$ (respectively, $\lweakmonocap(G)$) denote the minimum number of cops needed to see (resp. capture) a robber on $G$ using a weakly monotonic strategy.
\end{subsection}

\begin{subsection}{Hamming Graphs}

Our focus throughout the paper will be on \textit{Hamming graphs}, i.e. Cartesian products of complete graphs.  

\begin{defn}
The \textit{Cartesian product} of graphs $G$ and $H$, denoted $G \cart H$, is the graph with vertex set $G \times H$, where $(u,v)$ is adjacent to $(u',v')$ if and only if either
\begin{itemize}
\item $uu' \in E(G)$ and $v = v'$, or
\item $u = u'$ and $vv' \in E(H)$.
\end{itemize}
\end{defn}

\begin{defn}
For positive integers $d$ and $n$, the \textit{Hamming graph} $H(d,n)$ is the Cartesian product of $d$ copies of $K_n$.  We refer to $d$ as the \textit{dimension} of $H(d,n)$.
\end{defn}

We typically view the vertex set of $H(d,n)$ as the set of all ordered $d$-tuples with entries in $\{1, \dots, n\}$, where vertex $(x_1, \dots, x_d)$ is adjacent to $(y_1, \dots, y_d)$ provided that the two $d$-tuples differ in exactly one coordinate.\\
\end{subsection}
\end{section}

\begin{section}{Low-Dimensional Hamming Graphs}\label{sec:dimension_2}

In this section, we determine the $\ell$-visibility cop numbers for 1- and 2- dimensional Hamming graphs.  1-dimensional Hamming graphs are simply complete graphs and, as such, are fairly straightforward:

%*** WE SHOULD PROBABLY CITE SOMEONE FOR AT LEAST SOME OF THESE RESULTS -- THEY'VE SURELY BEEN PROVED ELSEWHERE ***
\begin{obs}\label{obs:simple}
For every positive integer $n$,
\begin{itemize}
\item \cite{Tan04} $c'_0(H(1,n)) = c_0(H(1,n)) = \ceil{n/2}$;
\item \cite{CCDDFM20} $c'_{\ell}(H(1,n)) = c_{\ell}(H(1,n)) = 1$ for all $\ell \ge 1$.
\end{itemize}
\end{obs}

2-dimensional Hamming graphs -- sometimes referred to as ``rook's graphs'' -- are more complex.   As usual, we view the vertex set of $H(2,n)$ as the set of ordered pairs $(i,j)$ with $i,j \in \{1, \dots, n\}$; for fixed $i$ and $j$, we refer to the set of vertices $(i,k)$ with $k \in \{1, \dots, n\}$ as ``column $i$'' and to the set of vertices $(k,j)$ with $k \in \{1, \dots, n\}$ as ``row $j$''.  We say that a cop \textit{sees} a row or column of $H(2,n)$ if that cop sees every vertex in that row or column.

For the original model of Cops and Robbers, Neufeld and Nowakowski showed in \cite{NN98} that $c(H(2,n)) = 2$.  Since $H(2,n)$ has diameter 2, it follows that whenever $\ell \ge 2$, we have $\lloc(H(2,n)) = 1$ and  $\lcap(H(2,n)) = c(H(2,n)) = 2$; hence we need only investigate the cases $\ell = 0$ and $\ell = 1$.  We begin with the latter.  

% When restricting to weakly monotonic strategies, things are straightforward.

% \begin{proposition}\label{prop:two_dim_vis_one_monotone}
% For all $n \ge 2$, we have $\weakmonoloc_1(H(2,n)) = \ceil{\frac{n}{2}}$.
% \end{proposition}
% \begin{proof}
% The upper bound is clear, since $\ceil{n/2}$ cops can begin the anywhere in rows $1, 2, \dots, \ceil{n/2}$ and move to see the remaining $\floor{n/2}$ rows in round 1.  Thus, $\weakmonoloc_1(H(2,n)) \le \ceil{n/2}$.

% For the lower bound, note that in a fully-cleaned row, recontamination is assured unless all columns are fully clean, i.e. unless the game is over.  Each row containing a cop is necessarily fully-cleaned, so in a weakly monotonic strategy, every row that has been seen by a cop must continue to be seen in every round until the end of the game.  Thus, in order to win the game, the cops must be able to see all $n$ rows in a single round; this requires at least $\ceil{n/2}$ cops, hence $\weakmonoloc_1(H(2,n)) \ge \ceil{n/2}$.
% \end{proof}

\begin{theorem}\label{thm:clique_product_search}
For all positive integers $n$, we have $c'_1(H(2,n)) = \ceil{\frac{n+1}{3}}$. 
\end{theorem}
\begin{proof}
The cases $n=1$ and $n=2$ are clear by inspection, so assume henceforth that $n \ge 3$.  

To simplify the presentation of the proof, we will view the process of ``recontamination'' slightly differently from usual.  In particular, during the recontamination phase of each round, we say that every vertex adjacent to a dirty vertex becomes recontaminated, \textit{even if that vertex is currently seen by a cop}.  Such vertices will, of course, immediately be seen (and thus cleaned) at the beginning of the next round, so this has no impact of the outcome of the game -- any winning cop strategy in the original formulation of the game is also a winning strategy in this new formulation, and vice-versa.

Let $k = \ceil{(n+1)/3}$.  To show that $c'_1(H(2,n)) \le k$, we give a strategy for $k$ cops to locate the robber.  Before presenting the cops' strategy, we make some useful observations.  First, note that during the recontamination phase of each round, if a row or column is not completely clean, then all vertices in that row or column become recontaminated.  Hence, it will be important to focus on how many rows and columns the cops have been able to fully clean in any given round.  For positive integers $a$ and $b$, define \ $S_{a,b} = \{(x,y) \in V(H(2,n)) \, \vert \, 1 \le x \le b \text{ and } 1 \le y \le a\}$.  Suppose that, immediately after a cop move, the cops have fully cleaned $a$  rows and $b$ columns.  By symmetry of $H(2,n)$, we may re-index rows and columns so that the fully cleaned rows are precisely rows 1 through $a$ and the fully cleaned columns are 1 through $b$.  Note also that every row or column that contains a cop is fully cleaned; thus, if the cops all occupy distinct rows and distinct columns (as will be the case in the cop strategy we will present), then we may further assume that the rows and columns have been re-indexed so that the cops occupy vertices $(1,1), \dots, (k,k)$.  We refer to this re-indexing process as \textit{normalizing} the game configuration.  After normalization and the subsequent recontamination phase, the set of clean vertices will be precisely $S_{a,b}$.  Broadly, the cops' strategy will be to clean a set of vertices of the form $S_{a,b}$, with the values of $a$ and $b$ gradually increasing over time.

\begin{figure}[h!]
    \centering
    \includegraphics[width=0.5\linewidth]{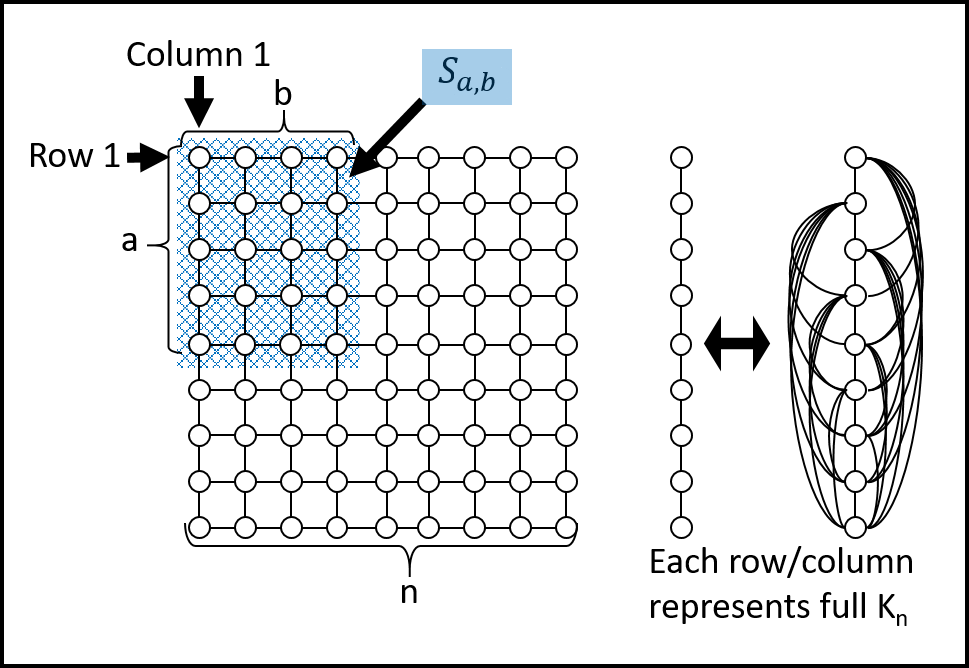}
    \caption{A depiction of $H(2,n)$}
    \label{fig:2d_hamming_legend}
\end{figure}

The cops begin the game on vertices $(1,1), (2,2), \dots, (k,k)$.  At the beginning of the cops' first turn, the cops can see all of rows $1$ through $k$ and columns $1$ through $k$.  If $n \le 2k$, then the cops can now move to see columns $k+1, \dots, n$, thereby fully cleaning the graph and winning the game, so suppose $n > 2k$.  The cops now move to vertices $(1,k+1), (2,k+2), \dots, (k,2k)$; after recontamination and normalization, at the start of the cops' next turn, the set of clean vertices is $S_{2k,k}$.

The cops now seek to gradually enlarge the set of clean vertices.  We claim that if, at the end of a round, the set of clean vertices is $S_{2k,c}$ for some positive integer $c$, then in the course of two rounds, the cops can either clean all vertices, or else cause the set of clean vertices to be $S_{2k,c'}$ where $c' > c$.  It would follow that by iterating this strategy, the cops can eventually clean the entire graph and thereby locate the robber.

Suppose that at the end of some round, the set of clean vertices is $S_{2k,c}$.  Let $\ell = n-2k$ and note that since $2k < n < 3k$ we must have $0 < \ell < k$.  On the cops' turn, the cops on vertices $(1,1), \dots, (\ell,\ell)$ move to vertices $(1,2k+1), \dots, (\ell, 2k+\ell) = (\ell,n)$; the remaining cops move from $(\ell+1, \ell+1), \dots, (k,k)$ to $(n-k+\ell+1,\ell+1), \dots, (n,k)$.  The set of clean vertices was initially $S_{2k,c}$ and the cops have now additionally cleaned rows $2k+1$ through $n$, so the first $c$ columns are fully clean; moreover, the cops have also cleaned the last $k-\ell$ columns.  If $c+k-\ell \ge n$, then the cops have cleaned all columns and have thus located the robber; otherwise, they have fully cleaned $c+k-\ell$ columns.  The cops occupied the first $k$ rows prior to moving and subsequently cleaned the last $\ell$; if $k+\ell \ge n$ then they have cleaned all rows and located the robber, and otherwise they have fully cleaned a total of $k+\ell$ rows.  Thus, either the cops have won, or the set of clean vertices after recontamination and normalization will be $S_{k+\ell,c+k-\ell}$, i.e. $S_{n-k, c+3k-n}$.  

In the next round, the cops move from $(1,1), \dots, (k,k)$ to $(1,n-k+1), \dots, (k,n)$.  This ensures that they fully clean the first $c+3k-n$ columns, the first $k$ rows, and the last $k$ rows, so either the cops have fully cleaned the graph, or the set of clean vertices after recontamination and normalization will be $S_{2k, c+3k-n}$.  (See Figure \ref{fig:dimension_2_visibility_1}.)  Since $3k > n$ we have $c+3k-n > c$, so the cops have successfully enlarged the set of clean vertices, as claimed; by repeating this strategy, the cops eventually locate the robber.

\begin{figure}[h!]
    \centering
    \includegraphics[width=1.0\linewidth]{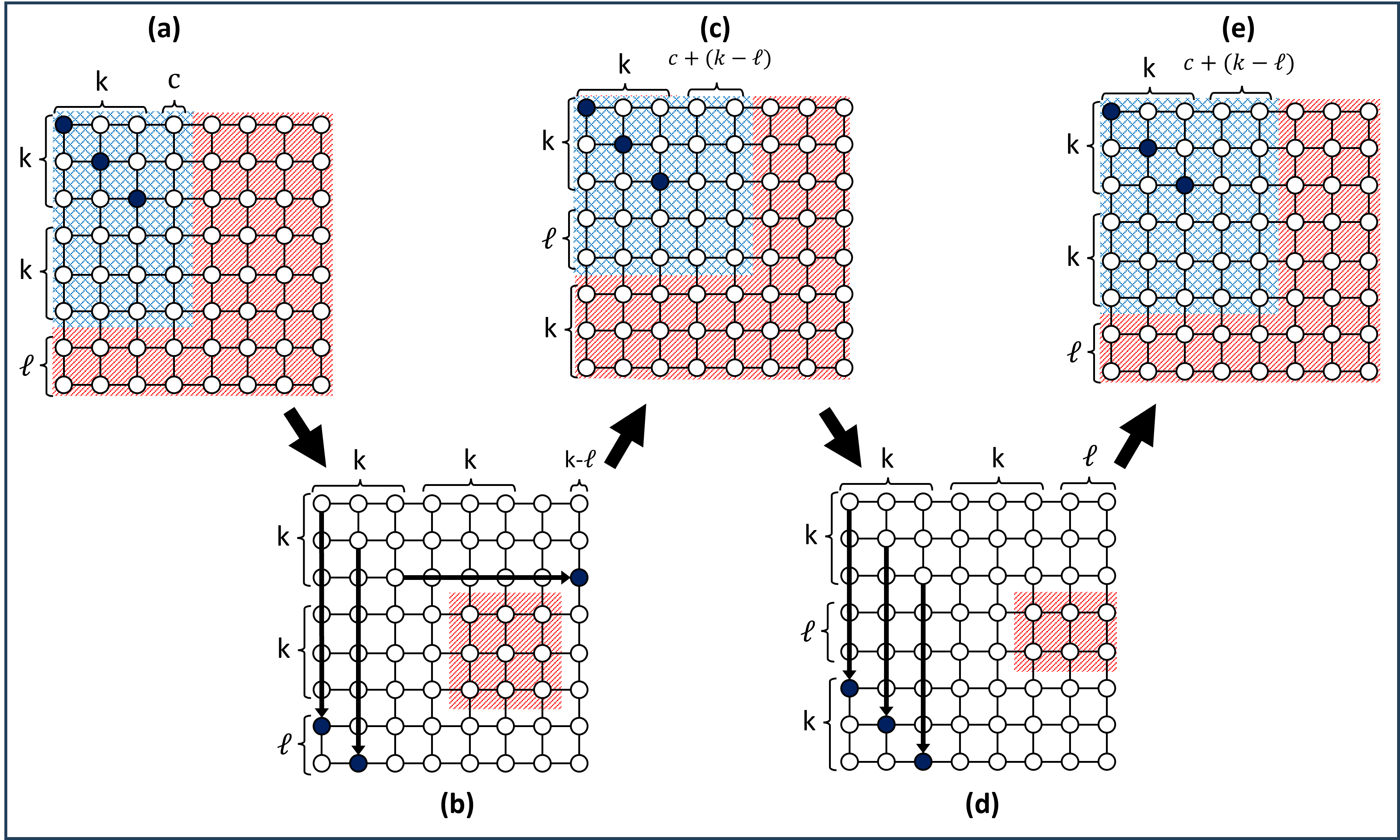}
    \caption{Example move to expand $S_{2k,c}$}
    \label{fig:dimension_2_visibility_1}
\end{figure}

We next establish the lower bound.  It suffices to show that $\floor{n/3}$ cops cannot locate the robber, so let $k = \floor{n/3}$ and consider the game with $k$ cops.  As noted above, we will make heavy use of the observation that during the recontamination phase of each round, when a row or column is not completely clean, all vertices in that row or column become recontaminated.  Hence, after each recontamination phase (and suitable re-indexing of the rows and columns), the set of clean vertices will be $S_{a,b}$ for some $a$ and $b$.

On the cops' first move, suppose that $\ell$ of the cops move vertically (i.e. move to new rows) for some $\ell$, and the remaining $k-\ell$ move horizontally or remain in place.  The cops cleaned at most $k$ different rows prior to moving and at most $\ell$ more afterward, so at most $k+\ell$ rows are clean at the end of the cops' turn.  Similarly, at most $2k-\ell$ columns are clean at the end of the cops' turn.  Thus, after recontamination and re-indexing of rows and columns, the set of clean vertices must be contained within the set $S_{k+\ell,2k-\ell}$.  We claim that at the end of every round for the remainder of the game, the set of clean vertices will have the form $S_{a,b}$ for positive integers $a$ and $b$ with $a \le 2k$, $b \le 2k$, and $a+b \le 3k$.  Since both $a$ and $b$ must always be strictly positive, the requirement that $a+b \le 3k \le n$ would imply that neither $a$ nor $b$ can exceed $n-1$; thus the cops can never fully clean the graph.  

Assume that at the end of some round, the set of clean vertices is $S_{a,b}$ where $a \le 2k$, $b \le 2k$, and $a+b \le 3k$.  Let $\ell$ denote the number of cops who move vertically in the cops' next turn.  We consider two cases.  

\textbf{Case 1}: $\ell \ge n-b$ or $k-\ell \ge n-a$.  By symmetry, we may suppose $\ell \ge n-b$.  The restrictions $n-b \le \ell \le k$ and $b \le 2k$ together imply 
$$3k \ge 2k+\ell \ge 2k+n-b \ge n.$$  
By choice of $k$ we have $3k \le n$, so we must in fact have equality throughout the chain of inequalities above; in particular, we have $n=3k$ and $\ell = k$, so in fact no cops move horizontally.  At the beginning of the round, the set of clean vertices was $S_{a,b}$, so the last $n-a$ rows were completely dirty and contained no cops.  Since no cops move horizontally, no new columns are seen as a result of the cops move.  Hence, if the cops do not move into all of the last $n-a$ rows, then the only columns that will be clean after the cops' move are those columns that contained cops at the start of the turn; there are at most $k$ of these.  Likewise, the only fully clean rows will be those rows that contained cops at the prior to the cops' move and those rows to which the cops subsequently moved; there are at most $2k$ of these.  Thus, after recontamination and re-indexing of rows and columns, the set of clean vertices will be contained in $S_{2k,k}$; it follows that the set of clean vertices will have the form $S_{a',b'}$ with $a' \le 2k$, $b' \le 2k$, and $a'+b' \le 3k$, as claimed.

\textbf{Case 2}: $\ell < n-b$ and $k-\ell < n-a$.  As in Case 1, at the beginning of the round, the last $n-a$ rows were completely dirty and contained no cops.  Since $\ell < n-a$, the cops cannot clean all of these rows.  Thus, as in Case 1, only two types of columns are clean after the cops' move:  columns that contained cops at the start of the turn and columns to which the cops have moved; there are at most $k+(k-\ell)$ such columns.  Similarly, since $k-\ell < n-b$, the only rows that are clean after the cops' move are the rows that contained cops at the start of the turn and the rows to which the cops have moved; there are at most $k+\ell$ of these.  Hence, after recontamination and re-indexing of rows and columns, the set of clean vertices will be contained within $S_{k+\ell,2k-\ell}$; it again follows that the set of clean vertices will have the form $S_{a',b'}$ with $a' \le 2k$, $b' \le 2k$, and $a'+b' \le 3k$, as claimed.

We have thus established the claim that at the end of each round the set of clean vertices will be $S_{a,b}$ where $a+b \le 3k$, hence the cops can never locate the robber.
\end{proof}

As noted above, Neufeld and Nowakowski \cite{NN98} showed that $c(H(2,n)) = 2$, so in light of Theorem \ref{thm:clique_product_search}, one might expect that it is typically much more difficult to see the robber on $H(2,n)$ than to capture him.  Quite surprisingly, this is not the case: capturing a robber in the 1-visibility game is substantially more difficult than seeing him, as we next show.  

Our proof of this fact will employ a useful result from \cite{CCDDFM20}.  Given a graph $G$ and a subgraph $H$ of $G$, a \textit{retraction} from $G$ to $H$ is a homomorphism from $G$ to $H$ that maps each vertex of $H$ to itself; that is, it is a map $\phi: V(G) \rightarrow V(H)$ such that $uv \in E(G) \Rightarrow \phi(u)\phi(v) \in E(H)$ and $\phi(v) = v$ for all $v \in V(H)$. We say that $H$ is a \textit{retract} of $G$ if there exists a retraction from $G$ to $H$.   Clarke et al. showed that -- as is the case with many variants of Cops and Robbers -- both $\lloc$ and $\lcap$ are monotone with respect to retraction:

\begin{theorem}[\cite{CCDDFM20}, Theorem 2.1]\label{thm:monotone_retract}
For any $\ell \ge 1$, any graph $G$, and any retract $H$ of $G$, we have $\lloc(H) \le \lloc(G)$ and $\lcap(H) \le \lcap(G)$.
\end{theorem}

We are now ready to determine $c_1(H(2,n))$.

\begin{theorem}\label{thm:clique_product_capture}
For all $n \ge 2$, we have $c_1(H(2,n)) = \ceil{\frac{n+1}{2}}$.
\end{theorem}
\begin{proof}
The case $n = 2$ is clear by inspection, so suppose $n \ge 3$.  

We begin by arguing that $c_1(H(2,n)) \le \ceil{\frac{n+1}{2}}$.  Suppose first that $n$ is odd, and let $n=2k+1$ for some $k$; we give a strategy for $k+1$ cops to capture the robber.  Initially, the cops begin the game on vertices $(1,1), (1,2), \dots, (1,k+1)$.  At the beginning of the cops' next turn, if the robber is in rows 1 through $k+1$ (or in column 1), then the cops can capture him immediately, so suppose otherwise.  The cops move to vertices $(1,k+1), (1,k+2), \dots, (1,2k+1)$.  Since the robber was in one of the last $k$ rows, some cop sees him; by symmetry, we may suppose the robber was seen on vertex $(2k+1,2k+1)$.

On the robber's turn, the robber cannot stay in row $2k+1$, lest he be captured by the cop on $(1,2k+1)$.  Similarly, he cannot move to $(i,2k+1)$ for any $i \in \{k+1,\dots, 2k\}$.  Hence, the robber must move to vertex $(i,2k+1)$ for some $i \in \{1, \dots, k\}$.  The cop on $(1,2k+1)$ now moves to $(2k+1,2k+1)$, while the cops on $(1,k+1), \dots, (1,2k)$ move to $(1,1), \dots, (1,k)$.  After the cops' move, the cop on $(2k+1,2k+1)$ will be in the same column as the robber, while one of the other cops will be in the same row as the robber.  Hence, wherever the robber moves on his next turn, he will still be in the same row or column as some cop.  Consequently, some cop will see the robber at the beginning of the cops' next turn and can subsequently capture him.

Now suppose that $n$ is even, and let $n = 2k+2$ for some $k$; we claim that $k+2$ cops can capture the robber.  One cop begins the game on $(2k+2,2k+2)$.  If the robber ever moves to row $2k+2$ or column $2k+2$, then that cop can capture him.  Otherwise, the cop remains on $(2k+2,2k+2)$ for the remainder of the game.  Since the robber must stay within the first $2k+1$ rows and $2k+1$ columns, the remaining $k+1$ cops can play a winning strategy to capture the robber on $K_{2k+1}\cart K_{2k+1}$ as detailed above.

We will next show that $c_1(H(2,n)) \ge \ceil{\frac{n+1}{2}}$.  Suppose first that $n$ is even, and let $n=2k$ for some $k$; we will give a strategy for the robber to evade $k$ cops.  We claim that if, at the beginning of a robber turn, the robber's current row or current column (or both) has no cops, then the robber can move so as to guarantee that after the next cop turn, the robber's new row or new column (or both) will have no cops.  In particular, this implies that the robber will not have been captured; hence the robber can repeat this strategy perpetually, thereby forever evading capture.

So, suppose that it is the robber's turn and that the robber's row or column (or both) has no cops.  By symmetry, we may suppose that the robber's row has no cops.  We may also assume, by symmetry, that the cops are all within the first $k$ rows and the first $k$ columns of the graph.  
%Since the robber's row has no cops, he may move (without being captured) to any vertex of his row in the last $k$ columns.  
Recall that to win, the cops are required to capture the robber with probability 1 -- i.e. they cannot ``get lucky'' -- so we may suppose that the robber knows, in advance, to which vertices the cops intend to move.  If no cop intends to move to the robber's row, then he moves to any vertex of his row in the last $k$ columns; this ensures that at the beginning of the next robber turn, the robber's row will have no cops.  On the other hand, if some cop does intend to move to the robber's row, then that cop must move vertically.  Consequently, at most $k-1$ cops may move horizontally, so after the cops move, at most $k-1$ of the last $k$ columns will contain cops.  The robber moves to any column that will not contain a cop, thereby ensuring that at the beginning of the robber's next turn, his column will not contain a cop.  This completes the proof of the lower bound when $n$ is even.

Suppose now that $n$ is odd, and let $n=2k+1$ for some $k$.  It is easily seen that $K_{2k} \cart K_{2k}$ is a retract of $K_{2k+1} \cart K_{2k+1}$ under the map $\phi : V(K_{2k+1} \cart K_{2k+1}) \rightarrow V(K_{2k} \cart K_{2k})$ defined by $\phi((x,y)) = (\min\{x,2k\}, \min\{y,2k\})$.  Thus, by Theorem \ref{thm:monotone_retract} and the argument above, we have $c_1(K_{2k+1} \cart K_{2k+1}) \ge c_1(K_{2k} \cart K_{2k}) = k$.
\end{proof}

Next, we return to the problem of determining $c_0(H(2,n))$.  In the 0-visibility game, the cops capture the robber when, and only when, they see him; hence $c_0(G) = c'_0(G)$ for all $G$.  We actually determine $\weakmonoloc_0(H(2,n))$ in addition to $c'_0(H(2,n))$, because the optimal search strategy for the cops can be carried out in a weakly monotonic manner.

\begin{theorem}\label{thm:dimension_2_visibility_0}
For all positive integers $n$, we have $c'_0(H(2,n)) = \weakmonoloc_0(H(2,n)) = \ceil{\frac{n^2+n}{4}}$.
\end{theorem}

\begin{proof}
The claim is clear by inspection when $n \le 2$, so assume $n \ge 3$.  Let $G = H(2,n)$.  Since $c'_0(G) \le \weakmonoloc_0(G)$, it suffices to prove that $\weakmonoloc_0(G) \le \ceil{\frac{n^2+n}{4}}$ and $c'_0(G) \ge \ceil{\frac{n^2+n}{4}}$.  

%As in the proof of Theorem \ref{thm:mono_product_cliques}, 
Throughout this proof we will consider each vertex to be in one of three states:
\begin{itemize}
\item a vertex is \textit{undiscovered} if it has not yet been cleaned at any point during the game;  
  % Subtle distinction between "dirty" and "contaminated" -- vulnerable vertices can temporarily become dirty so long as they are immediately cleaned, but dirty vertices have always been dirty
\item a vertex is \textit{vulnerable} if it has been cleaned but has at least one undiscovered neighbor; and
\item a vertex is \textit{secure} if it has been cleaned and has no undiscovered neighbors.
\end{itemize}
Note that there is a subtle distinction between ``dirty'' and ``undiscovered'' vertices: an undiscovered vertex is and always has been dirty, but a vulnerable vertex may briefly become dirty without violating weak monotonicity of the cops' strategy, provided that the vertex is cleaned on the ensuing cop turn.  
%That is, all undiscovered vertices are dirty, but not all dirty vertices are undiscovered.  
In any weakly monotonic cop strategy, if a vertex is vulnerable at the end of some round $k$, then it must be seen at some point during round $k+1$: either it is seen at the end of round $k$ (and hence prior to the cop move in round $k+1$), or it becomes recontaminated in round $k$ and must be seen after the cop move in round $k+1$ to maintain weak monotonicity.  In addition, secure vertices can never again be recontaminated, since all of their neighbors are either secure or vulnerable (and, in either case, are clean immediately prior to each recontamination phase).  Thus secure vertices remain clean for the remainder of the game, while vulnerable vertices remain vulnerable until they eventually become secure.

To establish the upper bound on $\weakmonoloc_0(G)$, we present a weakly monotonic strategy for $\ceil{\frac{n^2+n}{4}}$ cops to search $G$.  The cops' strategy will call for many cops to focus on keeping vulnerable vertices clean; we call such cops \textit{maintenance} cops.  Cops that are not currently serving as maintenance cops are referred to as \textit{free}.  Loosely, the cops' strategy will ensure that at all times, the movements of the maintenance cops suffice to maintain weak monotonicity, while the free cops may move however they like among the secure and vulnerable vertices.

We envision the vertices of $G$ as being arranged into a $n\times n$ square where each row and each column corresponds to a copy of $K_n$.  We consider the ``top'' row and the ``leftmost'' column to be the first row and first column respectively, and the bottom row and rightmost column are the $n^{th}$ row and $n^{th}$ column respectively.  Note that a vertex is secure if and only if each vertex in the same row or column has been cleaned at some point.  Thus, by re-indexing the rows and columns if necessary, the set of secure vertices can always be viewed as the intersection between the first $a$ rows and and the first $b$ columns of $G$, for some $a$ and $b$.  In other words, we view the secure set as being in the ``upper-left corner'' of $G$.

We are now ready to explain the cops' strategy.  Loosely speaking, the cops will establish a set of secure vertices, initially consisting of a single vertex in the upper-left corner.  The cops will gradually expand the secure set until it contains all vertices in the graph, at which point $G$ has been fully cleaned and the cops have necessarily seen the robber.  We distinguish between two types of cop move: \textit{expansion moves} are those on which the secure set actually expands, while \textit{transition moves} are moves the cops make to prepare for the next expansion move.  An expansion move actually consists of the cop moves in two or three consecutive rounds, as it takes more than one round for the cops to expand the secure set while maintaining weak monotonicity.  If the secure set is of size $a \times b$ prior to the expansion move, then it will be of size $(a+1) \times b$ or $a \times (b+1)$ after the move; that is, each expansion move expands the secure set by a single row or column.  Whenever $a = b$, the next expansion move expands the secure set to size $(a+1)\times b$; otherwise, it expands the secure set to size $a \times (b+1)$.  
%This progression of moves will yield a secure set after any expansion move which will always be square $(a \times a)$ or close to square $((a+1) \times a)$.

%Throughout the proof, vertices within a given row (resp. column) will be referred to as ``even'' or ``odd'' vertices based upon the parity of their first (resp. second) coordinate.  Similarly, a row (resp. column) will be referred to as ``even'' or ``odd'' based upon the parity of its second (resp. first) coordinate.  
% The above definitions were only used once, in this paragraph, and have thus been removed.
To begin the game, the cops occupy vertices $1, 3, 5, \dots$ of the first column and vertices $2, 4, 6, \dots$ of the first row of $G$.  Note that this initial placement requires $n$ cops, all of which are considered to be maintenance cops; all remaining cops place themselves on vertex $(1,1)$ and are considered to be free cops.  On the first turn, the maintenance cops will move to cover every vertex on the first row and column that was not previously covered.  After this cop move, vertex $(1,1)$ is secure, all other vertices in the first row and first column are vulnerable, and all remaining vertices are undiscovered.

By definition, all non-secure neighbors of secure vertices are necessarily vulnerable; throughout the game, the cops will play so as to ensure that these are the only vulnerable vertices.  We refer to the $a \times (n-b)$ rectangle of vulnerable vertices that share a row with a secure vertex the \textit{upper vulnerable set}, while the $b \times (n-a)$ rectangle of vulnerable vertices that share a column with a secure vertex is the \textit{lower vulnerable set}; see Figure \ref{fig:KnxKn_weak}.
     
\begin{figure}[ht]
    \centering
    \includegraphics[width=0.5\linewidth]{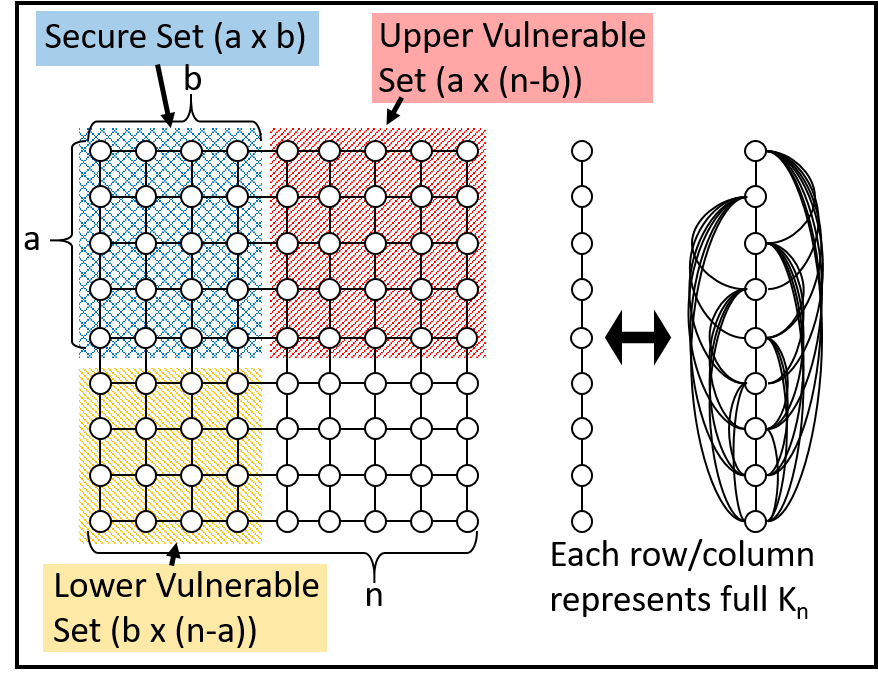}
    \caption{Secure and vulnerable sets}
    \label{fig:KnxKn_weak}
\end{figure}

We next describe how the maintenance cops play during transition moves so as to maintain weak monotonicity.  (Recall that the cops must see every vulnerable vertex in every round, either before or after the cop move.) The positioning and movement of cops in each vulnerable set will depend on the parity of the dimensions of the vulnerable set.  Suppose the secure set is of size $a \times b$, let $s^+ = \ceil{\frac{a}{2}}$, let $s^- = \floor{\frac{a}{2}}$, and let $t = \floor{\frac{n-b}{2}}$.  The cops will form a ``checkerboard pattern'' in the upper vulnerable set, as follows.

Initially, cops occupy columns $n-2t+1, n-2t+3, \dots, n-1$ of rows $1, 3, \dots, 2s^+-1$ and columns $n-2t+2, n-2t+4, \dots, n$ of rows $2, 4, \dots, 2s^-$.  In each round, each of these cops will move from a vertex $(x,y)$ to vertex $(x+1, y)$ if $x < n$ and to vertex $(n-2t+1,y)$ if $x = n$.  (See Figure \ref{fig:UVS_Main1}.)  Note that in doing so, the cops ensure that, during each round, they see every vertex in the last $2t$ columns of the upper vulnerable set during each round.  If $n-b$ is odd, then the cops must also ensure that they see every vertex of the first column of the upper vulnerable set, i.e. column $b+1$ of $G$.  To do this, cops initially occupy vertices $2, 4, \dots, 2s^-$ of column $b+1$; on each turn the cop on vertex $(b+1,y)$ moves to vertex $(b+1,y+1)$ if $y < 2s^-$ and to vertex $(b+1,1)$ if $y = 2s^-$.  (See Figure \ref{fig:UVS_Main2}.)  Additionally, if $a$ is odd, then one cop occupies and remains at vertex $(b+1,a)$, as shown in Figure \ref{fig:UVS_Main3}.  
%Performing the maintenance actions described above yield cops positioned in a ``checkerboard pattern'' where the rectangular vulnerable set can be envisioned as a checkerboard made up of red and black vertices, where every other turn the cops occupy the red vertices and on the alternative turns the cops occupy the black vertices (except for odd upper vulnerable sets, in which a stationary cop is always located in the lower left corner.

    \begin{figure}[ht]
     \centering
     \begin{subfigure}[b]{0.3\textwidth}
         \centering
         \includegraphics[width=\textwidth]{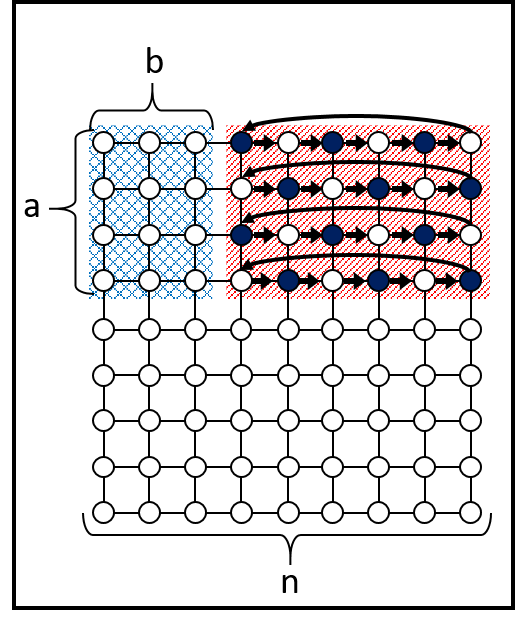}
         \caption{$(n-b)$ is even}
         \label{fig:UVS_Main1}
     \end{subfigure}
     \hfill
     \begin{subfigure}[b]{0.3\textwidth}
         \centering
         \includegraphics[width=\textwidth]{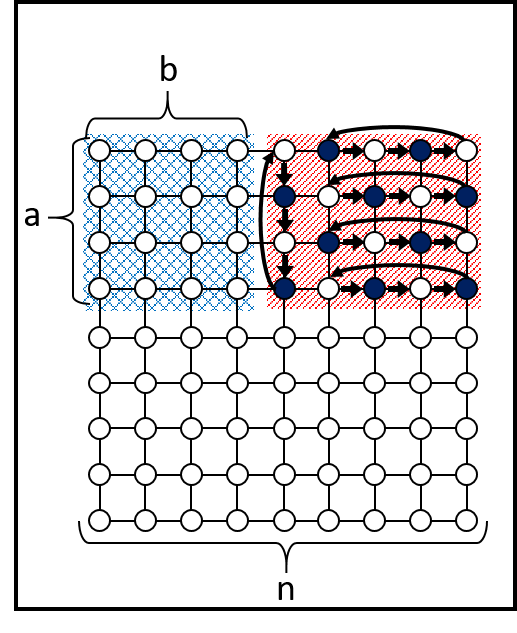}
         \caption{$(n-b)$ is odd; $a$ is even}
         \label{fig:UVS_Main2}
     \end{subfigure}
     % \vspace{8.mm}
     \hfill
     \begin{subfigure}[b]{0.3\textwidth}
         \centering
         \includegraphics[width=\textwidth]{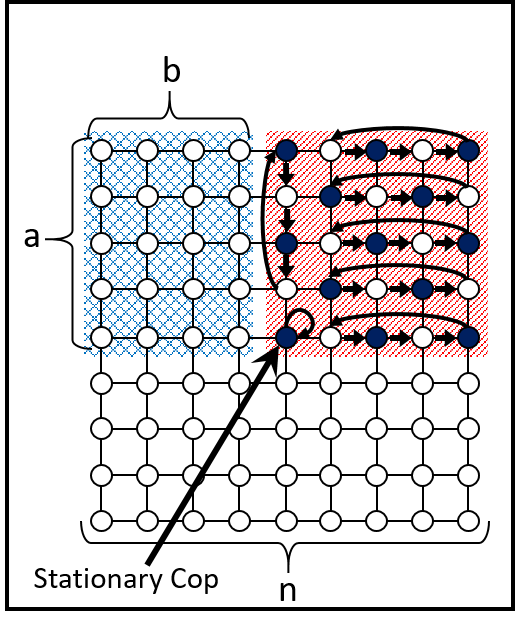}
         \caption{$(n-b)$ is odd; $a$ is odd}
         \label{fig:UVS_Main3}
     \end{subfigure}
     \hfill
        \caption{Cops' strategy for maintaining upper vulnerable set}
        \label{fig:UVS_Main}
\end{figure}

Additional maintenance cops use a symmetric strategy to patrol the lower vulnerable set, forming and moving in a checkerboard pattern, with a lone stationary cop in the upper-right corner if both $b$ and $n-a$ are odd; see Figure \ref{fig:LVS_Main}.

 \begin{figure}[ht]
     \centering
     \begin{subfigure}[b]{0.3\textwidth}
         \centering
         \includegraphics[width=\textwidth]{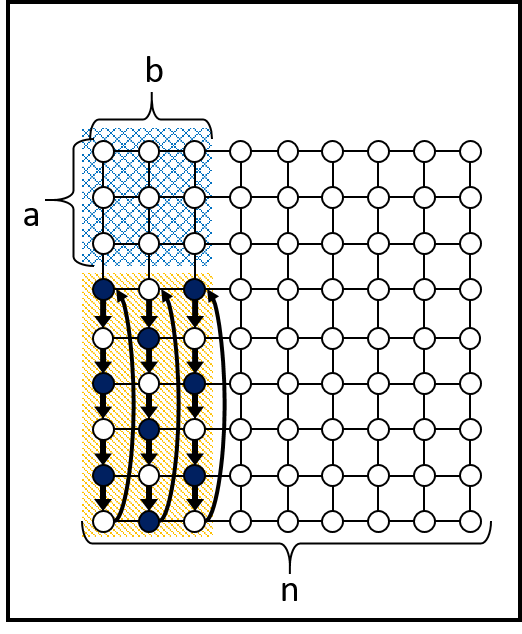}
         \caption{$(n-a)$ is even}
         \label{fig:LVS_Main1}
     \end{subfigure}
     \hfill
     \begin{subfigure}[b]{0.3\textwidth}
         \centering
         \includegraphics[width=\textwidth]{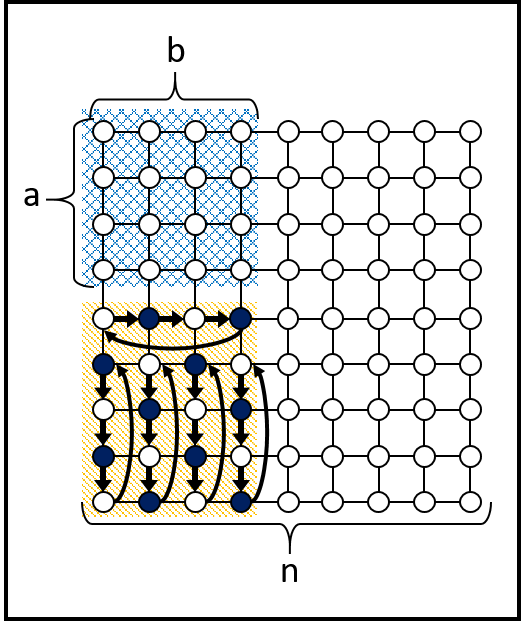}
         \caption{$(n-a)$ is odd; $b$ is even}
         \label{fig:LVS_Main2}
     \end{subfigure}
     % \vspace{8.mm}
     \hfill
     \begin{subfigure}[b]{0.3\textwidth}
         \centering
         \includegraphics[width=\textwidth]{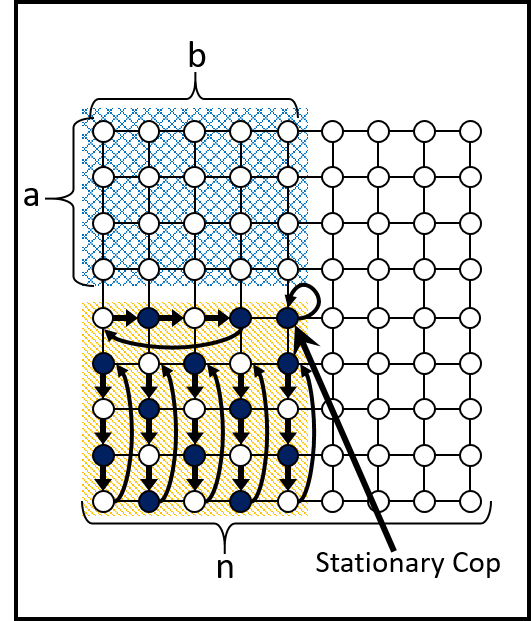}
         \caption{$(n-a)$ is odd; $b$ is odd}
         \label{fig:LVS_Main3}
     \end{subfigure}
     \hfill
        \caption{Cops' strategy for maintaining lower vulnerable set}
        \label{fig:LVS_Main}
\end{figure}

By moving thus, the cops ensure that every vertex of the upper vulnerable set is occupied on every other turn, except perhaps for a single vertex that is occupied every turn (in the case where the upper vulnerable set has an odd number of vertices).  Consequently, the number of maintenance cops needed for the upper vulnerable set is $\ceil{\frac{a\cdot (n-b)}{2}}$, while the number needed for the lower vulnerable set is $\ceil{\frac{(n-a)\cdot b}{2}}$.

Next, we explain how the cops execute expansion moves.  Recall that to prepare for expansion moves, free cops may need to move through the set of secure and vulnerable vertices; since the maintenance cops are preventing any recontamination, the free cops may take as long as they need to get themselves into position.  We first consider the case where the secure set has size $a \times b$ and the cops aim to expand it to $(a+1) \times b$.  To prepare for this expansion move, free cops move so that one free cop occupies the same vertex as each maintenance cop on row $a$ of the upper vulnerable set; we will refer to these free cops as \textit{expansion cops}.  Once all of the expansion cops are in position, they will move so as to clean row $a+1$.  If $n-b$ is even, then each expansion cop simply moves down to row $a+1$, thereby extending the cops' checkerboard pattern in the upper vulnerable set.  
%(Note that while the expansion cops move down, the maintenance cops in row $a$ move to the right, so the cops in row $a$ are offset from those in row $a+1$, as is needed for the checkerboard pattern.)  
The expansion cops then become maintenance cops, moving in concert with the other maintenance cops in the upper vulnerable set.  This ensures that all vertices in columns $b+1, \dots, n$ of row $a+1$ will be seen in every round; in particular, these vertices become vulnerable, and the vertices in columns $1, \dots, b$ of row $a+1$ become secure, thereby producing an $(a+1) \times b$ secure set.  (See Figure \ref{fig:Even_Expansion}.)  %Note that, in this case, $(n-b)/2$ expansion cops were needed.

\begin{figure}[h!]
     \centering
     \begin{subfigure}[b]{0.3\textwidth}
         \centering
         \includegraphics[width=\textwidth]{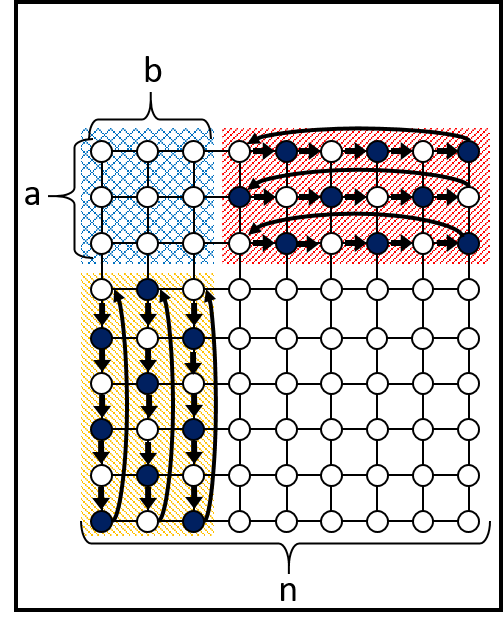}
         \caption{Before expansion} %Pre-Expansion Move
         \label{fig:Even_expan1}
     \end{subfigure}
     \hfill
     \begin{subfigure}[b]{0.3\textwidth}
         \centering
         \includegraphics[width=\textwidth]{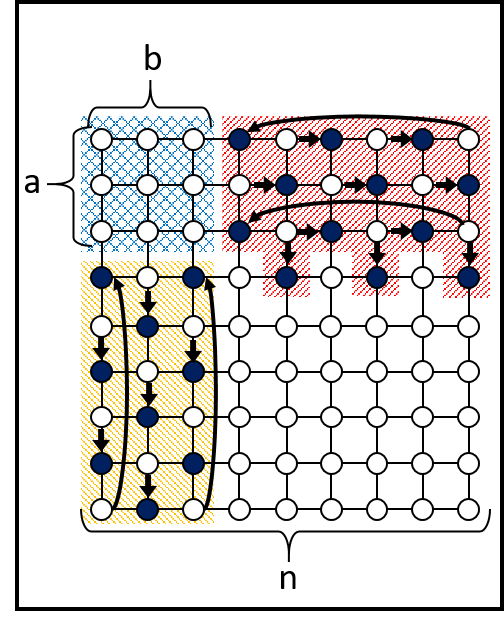}
         \caption{After first cop turn} %$1^{st}$ Cop-turn of Expansion Move
         \label{fig:Even_expan2}
     \end{subfigure}
     \hfill
     \begin{subfigure}[b]{0.3\textwidth}
         \centering
         \includegraphics[width=\textwidth]{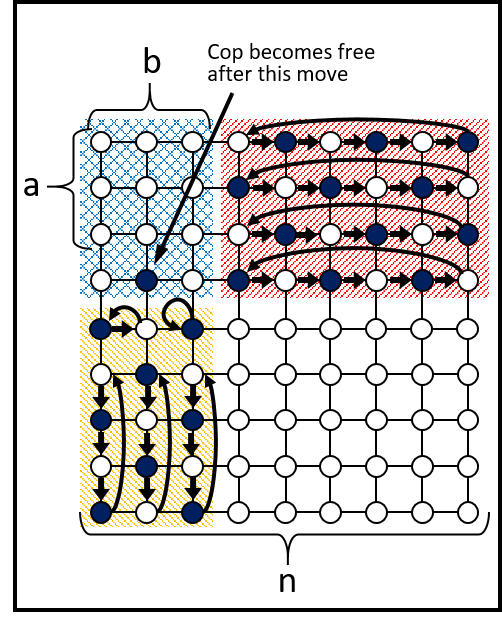}
         \caption{After expansion} %$2^{nd}$Cop-turn of Expansion Move
         \label{fig:Even_expan3}
     \end{subfigure}
        \caption{Expansion move ($n-b$ even)}
        \label{fig:Even_Expansion}
\end{figure}

If instead $n-b$ is odd, then the expansion process is slightly more complex.  If $n-b$ is odd and $a$ is even, then expansion cops occupy vertices $b+1, b+3, \dots, n$ of row $a$.  They wait until these vertices are simultaneously occupied by maintenance cops, at which point each expansion cop moves down to row $a+1$.  The expansion cop who moved to $(b+1,a+1)$ becomes a stationary maintenance cop, while the other expansion cops begin to cycle rightward, thereby extending the cops' checkerboard pattern in the upper vulnerable set.  (See Figure \ref{fig:Odd_NStat_Expansion}.)  %In this case, $(n-b+1)/2$ expansion cops were needed.  
Finally, suppose that both $n-b$ and $a$ are odd.  Recall that in this case, a stationary maintenance cop occupies vertex $(b+1,a)$.  Expansion cops move to occupy the vertices in columns $b+3, b+5, \dots, n$ of row $a$ and wait until maintenance cops move onto these vertices; then the expansion cops, along with the stationary maintenance cop on vertex $(b+1,a)$, all move down to row $a+1$.  The cop who moved to $(b+1,a+1)$ now begins to cycle vertically with the other maintenance cops in column $b+1$, while the remaining expansion cops cycle horizontally to extend the cops' checkerboard pattern in the upper vulnerable set.  (See Figure \ref{fig:Stat_Expansion}.)  
% \comment{*** THIS FIGURE NEEDS UPDATING -- it depicts expansion of the lower vulnerable set, not the upper one. ***}  %In this case, $(n-b-1)/2$ expansion cops were needed.

\begin{figure}[ht]
     \centering
     \begin{subfigure}[b]{0.3\textwidth}
         \centering
         \includegraphics[width=\textwidth]{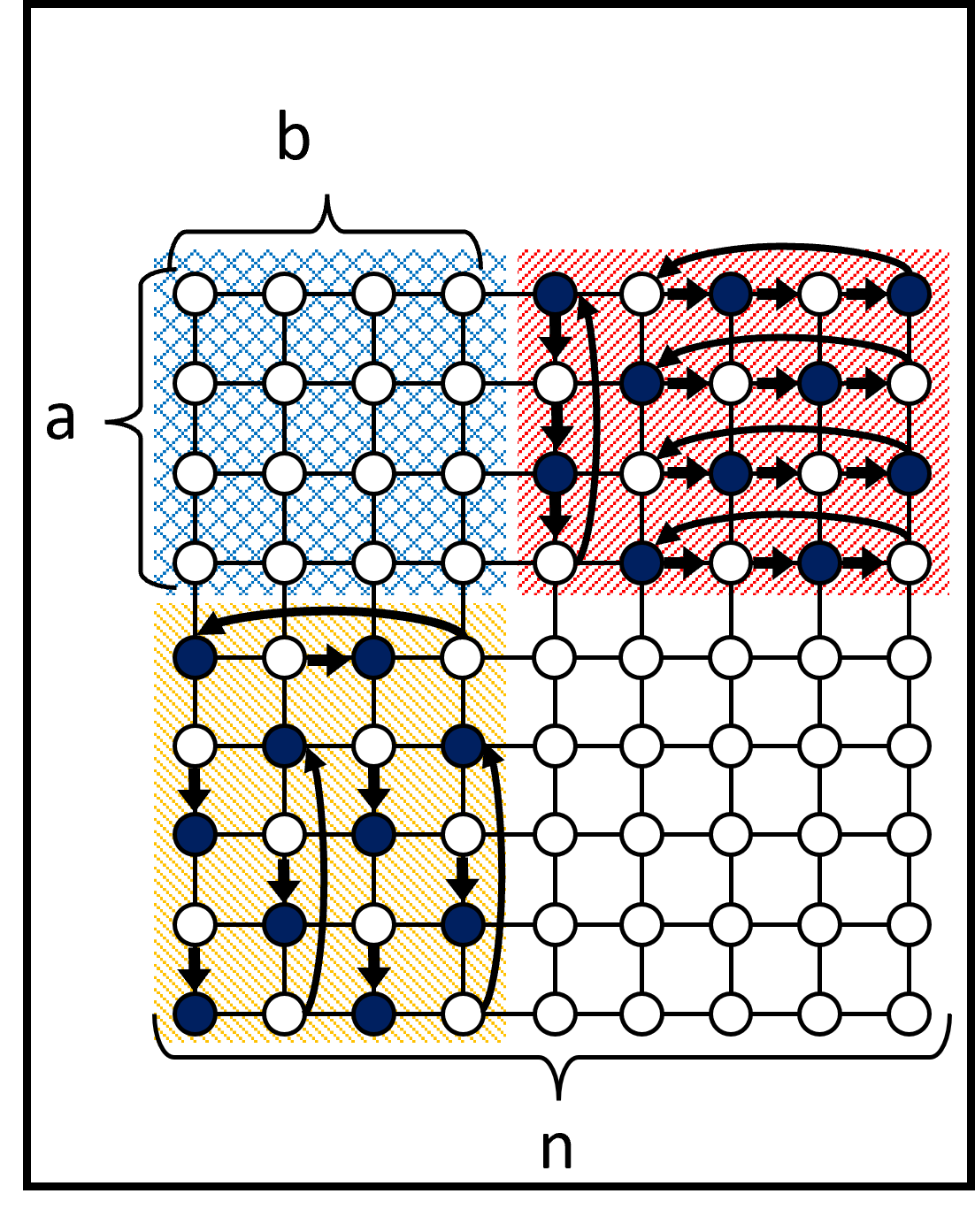}
         \caption{Before expansion}
         \label{fig:Odd_NStat_expan1}
     \end{subfigure}
     \hfill
     \begin{subfigure}[b]{0.3\textwidth}
         \centering
         \includegraphics[width=\textwidth]{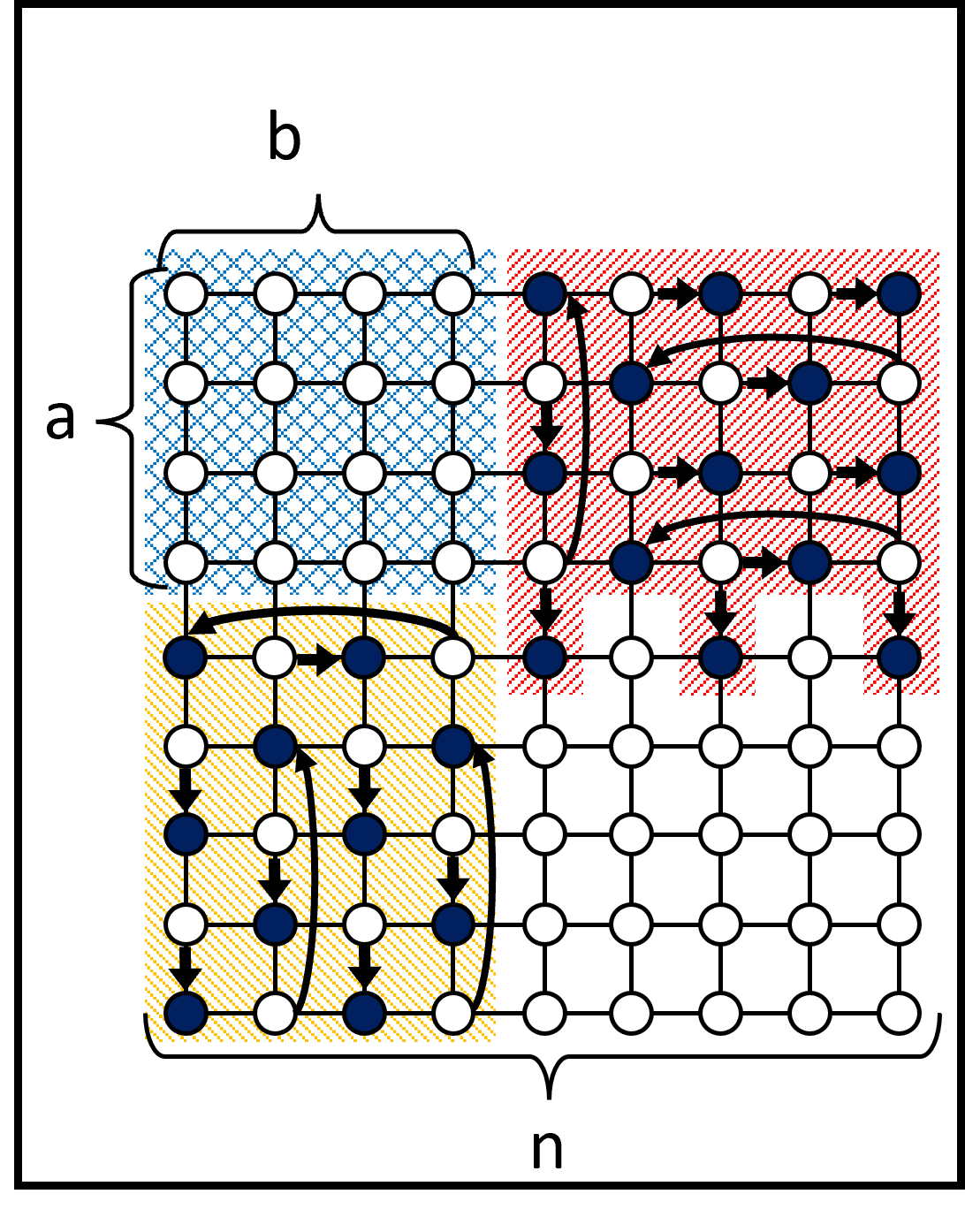}
         \caption{After first cop move} %$1^{st}$ Cop-turn of Expansion Move
         \label{fig:Odd_NStat_expan2}
     \end{subfigure}
     \hfill
     \begin{subfigure}[b]{0.3\textwidth}
         \centering
         \includegraphics[width=\textwidth]{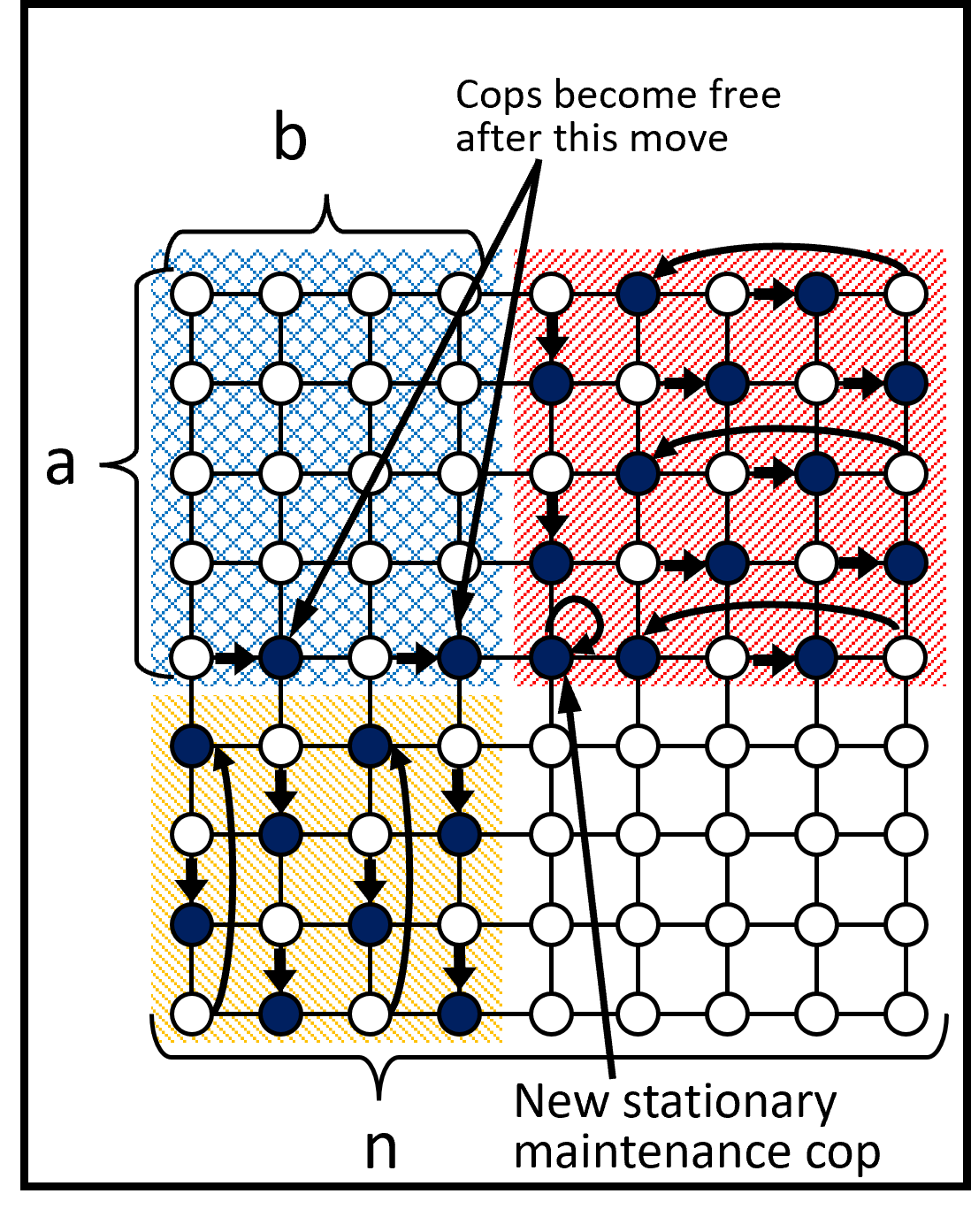}
         \caption{After expansion} %$2^{nd}$Cop-turn of Expansion Move
         \label{fig:Odd_NStat_expan3}
     \end{subfigure}
     % \hfill
     % \begin{subfigure}[b]{0.21\textwidth}
     %     \centering
     %     \includegraphics[width=\textwidth]{Graphics//Weak monotone on KnxKN/Odd_Non_Stat_Expansion_4.png}
     %     \caption{Post-Expansion Move}
     %     \label{fig:Odd_NStat_expan4}
     % \end{subfigure}
        \caption{Expansion move ($n-b$ odd; $a$ even)}
        \label{fig:Odd_NStat_Expansion}
\end{figure}

\begin{figure}[h]
     \centering
     \begin{subfigure}[b]{0.3\textwidth}
         \centering
         \includegraphics[width=\textwidth]{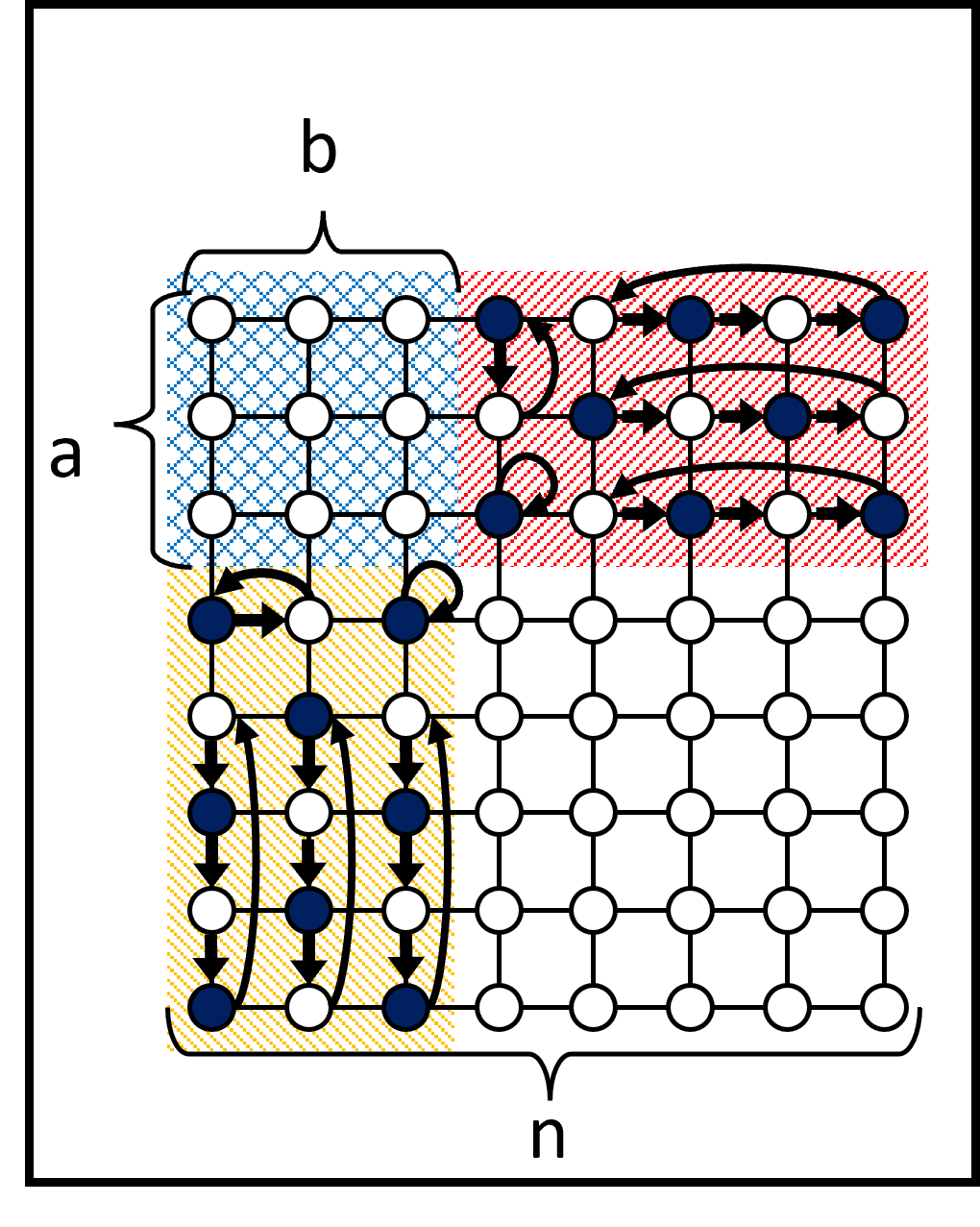}
         \caption{Before expansion} %Pre-Expansion Move
         \label{fig:Stat_expan1}
     \end{subfigure}
     \hfill
     \begin{subfigure}[b]{0.3\textwidth}
         \centering
         \includegraphics[width=\textwidth]{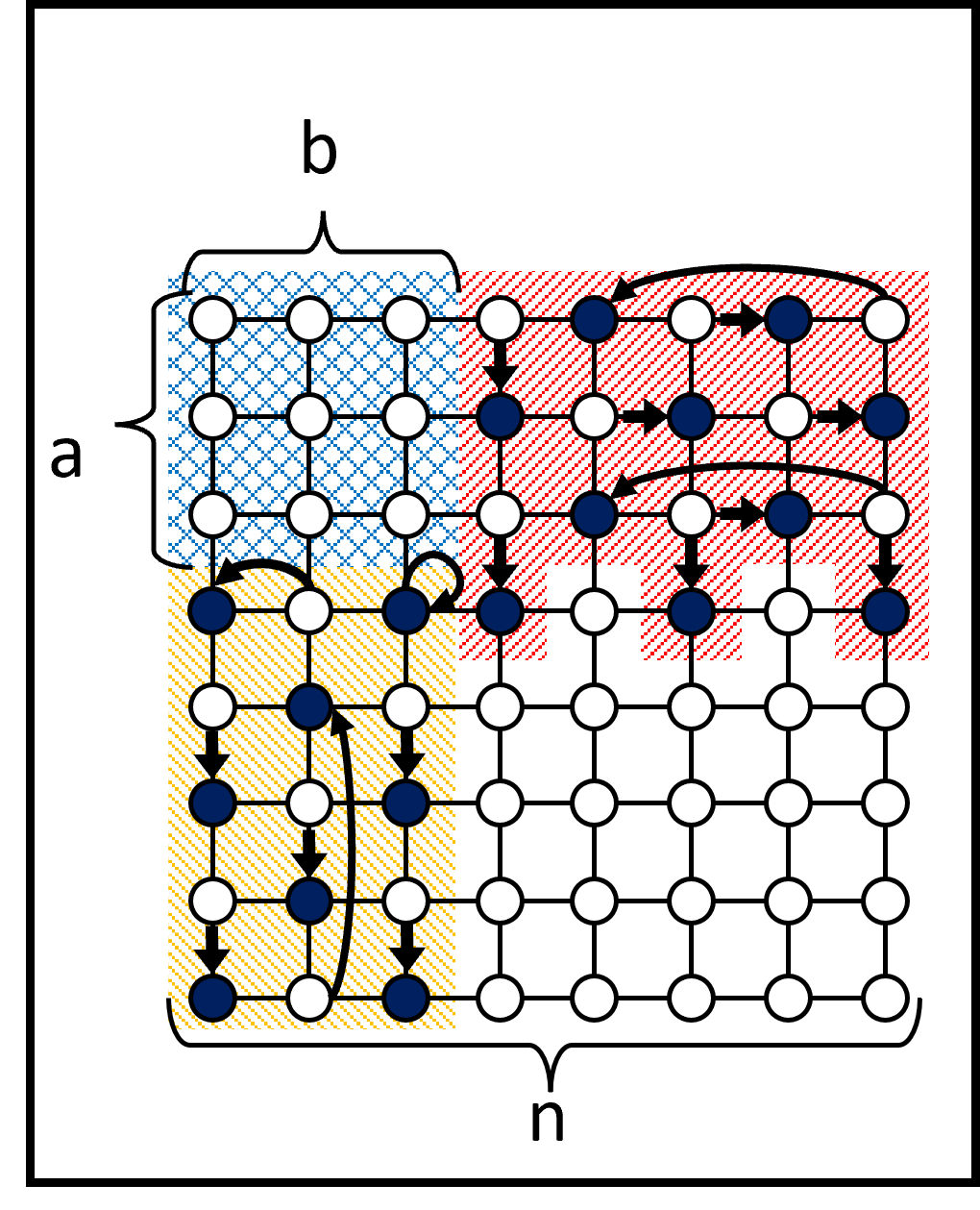}
         \caption{After first cop move} %$1^{st}$ Cop-turn of Expansion Move
         \label{fig:Stat_expan2}
     \end{subfigure}
     \hfill
     \begin{subfigure}[b]{0.3\textwidth}
         \centering
         \includegraphics[width=\textwidth]{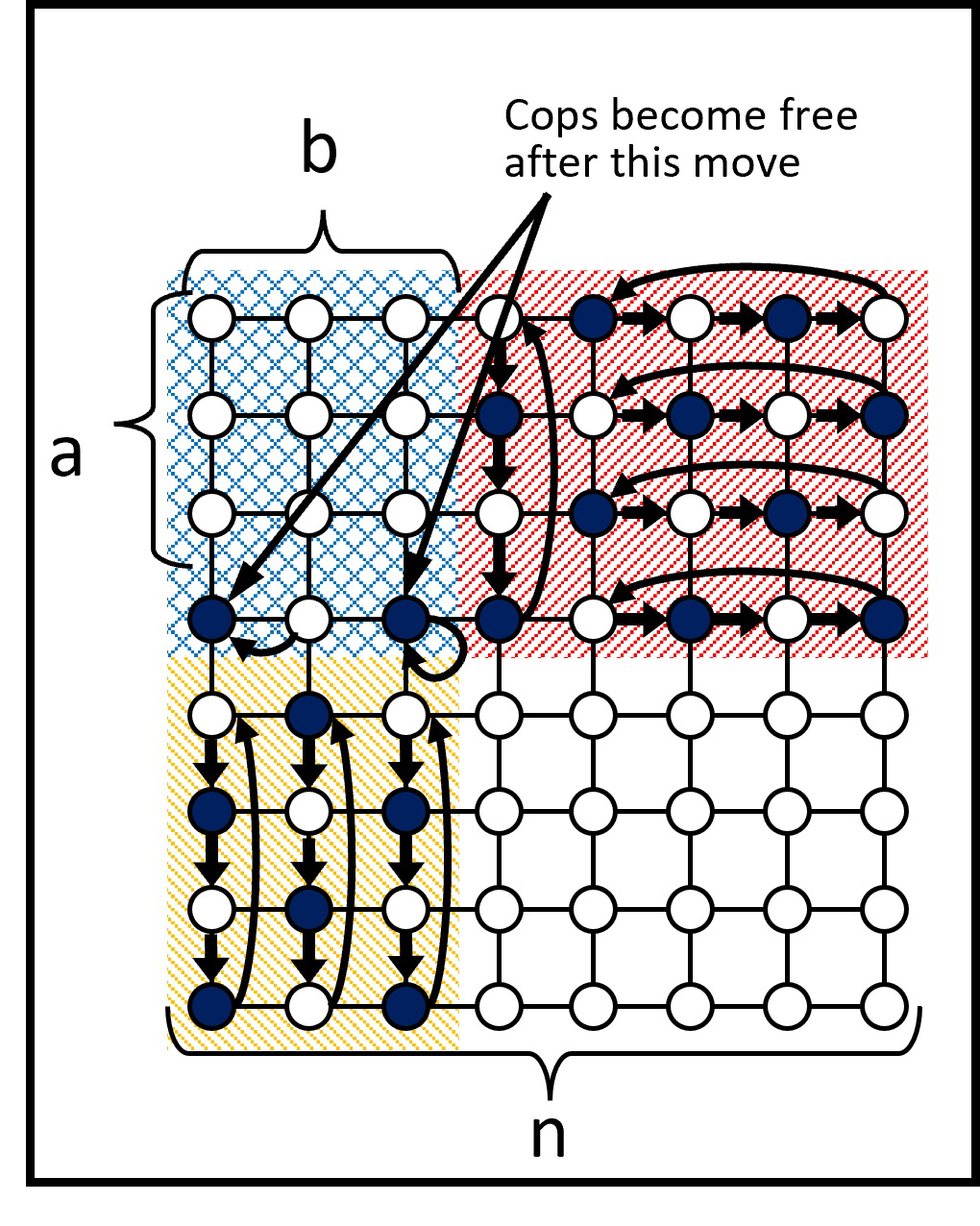}
         \caption{After expansion} %$2^{nd}$Cop-turn of Expansion Move
         \label{fig:Stat_expan3}
     \end{subfigure}
     % \hfill
     % \begin{subfigure}[b]{0.23\textwidth}
     %     \centering
     %     \includegraphics[width=\textwidth]{Graphics//Weak monotone on KnxKN/Stationary_Expansion_4.png}
     %     \caption{After expansion} %Post-Expansion Move
     %     \label{fig:Stat_expan4}
     % \end{subfigure}
        \caption{Expansion move ($n-b$ odd; $a$ odd)}
        \label{fig:Stat_Expansion}
\end{figure}

In any case, the total number of cops required to perform this expansion while also maintaining the upper vulnerable set is $\ceil{\frac{(a+1)(n-b)}{2}}$.  Meanwhile, another $\ceil{\frac{b(n-a)}{2}}$ maintenance cops patrol the lower vulnerable set.  Thus, overall, the expansion requires a total of $\ceil{\frac{(a+1)(n-b)}{2}} + \ceil{\frac{b(n-a)}{2}}$ cops.  Recall that expansion of the upper vulnerable set occurs when $a=b$, so the total number of cops required to perform this expansion move is $\ceil{\frac{(a+1)(n-a)}{2}}+\ceil{\frac{a(n-a)}{2}}$.  When $(n-a)$ is even, the total number of cops is $\frac{(a+1)(n-a)+a(n-a)}{2}$, which simplifies to $-a^2 + \left(n-\frac{1}{2}\right )a + \frac{n}{2}$; when $(n-a)$ is odd, the total number of cops is $\frac{(a+1)(n-a)+a(n-a)+1}{2}$, which simplifies to $-a^2 + \left(n-\frac{1}{2}\right)a + \frac{n}{2}+\frac{1}{2}$.  In either case, the total number of cops required is bounded above by $-a^2 + (n-\frac{1}{2})a + \frac{n+1}{2}$.  This bound is increasing in $a$ when $a < \frac{n}{2}$ and decreasing when $a > \frac{n}{2}$, so it is maximized for either $a=\floor{n/2}$ or $a=\ceil{n/2}$.  Upon examining both possibilities, we find that the expansion requires no more than $\frac{n^2 + n + 2}{4}$ cops.  Since the number of cops needed must be an integer, it is bounded above by $\floor{\frac{n^2+n+2}{4}}$ which, for integer values of $n$, is equivalent to $\ceil{\frac{n^2+n}{4}}$.

When $a=b+1$, the cops play similarly to expand the secure set from $a \times b$ to $a \times (b+1)$; the total number of cops needed for this expansion is $\lceil\frac{(b+1)(n-a)}{2}\rceil + \lceil\frac{a(n-b)}{2}\rceil$.  
%I.e. if $a>b$, then the cops will expand the secure set to once again be square.
     %Similar steps are taken when expanding the lower vulnerable set, yielding a total number of cops required in that case to be $\lceil\frac{(b+1)(n-a)}{2}\rceil + \lceil\frac{a(n-b)}{2}\rceil$.  Recall that this occurs when $a = b+1$.  
Similar computations to those used above show that the same upper bound applies: the expansion requires no more than $\ceil{\frac{n^2 + n}{4}}$ cops.
     %Using a similar approach as above the same upper bound applies, so we find that each of these moves can be performed using no more than $\lceil \frac{n^2 + n}{4} \rceil$ cops.  
Thus $\ceil{\frac{n^2 + n}{4}}$ cops can expand the secure set until it encompasses the entire graph, at which point the cops win.

For the lower bound, we show that $c'_0(H(2,n)) \ge \ceil{\frac{n^2+n}{4}}$, from which it will also follow that $\weakmonoloc_0(H(2,n)) \ge \ceil{\frac{n^2+n}{4}}$.  As a first step toward this lower bound, we bound the minimum size of a secure set that the cops must produce before winning the game.  First note that in any successful cop strategy, every non-secured vertex must be seen in the final round of the game.  Since each cop can only see a maximum of 2 vertices in each round, any strategy that does not create a secure set prior to the final round must use at least $\frac{n^2}{2}$ cops; hence any successful strategy using fewer than $\frac{n^2}{2}$ cops must create a secure set.

Recall that any secure set can be represented as a rectangular subset of $G$ and will induce upper and lower vulnerable sets as described above.  We determine the number of cops needed to expand the secure set from size $a \times b$ to size $(a+c) \times b$ in a single step for some $c \in \mathbb{N}$.
%Given a secure set of size $a \times b$, we consider the number of cops required to expand the set to be size $(a+c) \times b$, where $c \in \mathbb{N}$.  
Note that during the round in which the secure set is expanded, the cops must see every vulnerable or undiscovered vertex that is to be secure or vulnerable after expansion: at the beginning of the round, every such vertex either contains a cop or is dirty and hence must be cleaned by the end of the round.  Following expansion, there are $(a+c)\cdot b$ secure vertices and $(a+c)\cdot (n-b) + b \cdot (n-a-c)$ vulnerable vertices; of these, at most $a \cdot b$ were secure prior to expansion and thus did not need to be seen.  In total, the number of vertices that the cops must see in the round prior to expansion is at least $(a+c) \cdot b + (a+c) \cdot (n-b) + b \cdot (n-a-c) - ab$, which simplifies to $(a+c)(n-b) + b(n-a)$.

%The two vulnerable sets include $a(n-b) + b(n-a)$ vertices between them, and another $c(n-b)$ dirty vertices must be seen in the course of expansion.  \comment{*** We should explain why all of these vertices must be seen in a single round.  Remember that we are NOT assuming weak monotonicity here... ***}  

Since each cop can see at most two vertices on each turn, the number of cops needed is at least 
%$\lceil\frac{a(n-b)+b(n-a)+c(n-b)}{2}\rceil$, which simplifies to 
$\ceil{\frac{b(n-a) + (a+c)(n-b)}{2}}$.  Thus, for the purposes of obtaining a lower bound on the number of cops needed, we may suppose that $c=1$: if the number of cops available meets the lower bound on the number of cops needed to expand the secure set by two or more rows at once, then it also meets the lower bound on the number of cops needed to repeatedly expand the secure set by one row at a time, thereby achieving the same effect.  

Similarly, if the cops want to expand the secure set in multiple directions, i.e. to produce a secure set of size $(a+c) \times (b+d)$ for $c, d \in \mathbb{N}$, then at least $\ceil{\frac{(a+b)n -2ab + (n-b)c + (n-a-c)d}{2}}$ cops are required.  This is because the number of vulnerable vertices before expansion was $a(n-b) + b(n-a)$, another $n-b$ dirty vertices must be cleaned in each of the $c$ new rows of the secure set, and another $n-a-c$ dirty vertices must be cleaned in each of the $d$ new columns of the secure set. 
%the number of newly seen vertices in one of the $c$ rows of the expansion is $c(n-b)$, and the number of newly seen and not previously counted vertices from one of the $d$ columns of the expansion is $d(n-a-c)$.  
Comparing this to our lower bound on the number of cops needed to expand the secure set through a series of expansion moves, each consisting of only one direction of expansion, we see that expanding in multiple directions simultaneously does not result in a smaller lower bound.  Therefore we may assume that that each expansion move only expands the size of the secure set by one row or by one column.%, but not both.

Suppose there exists a winning strategy using $\ceil{\frac{n^2 + n}{4}} - 1$ cops.  Since each cop can see at most two vertices in the final round, the number of non-secure vertices at the beginning of the final round cannot exceed $2\ceil{\frac{n^2 + n}{4}} - 2$. Thus, prior to the cops' last turn, the secure set must contain at least $\frac{n^2 - n}{2}$ vertices.  Consequently, the cops must at some point produce an $a \times b$ secure set with $ab \ge \frac{n^2-n}{2}$; in particular, we must eventually have $a+b > n$.
     
Because we have assumed that the cops expand the set by one row or by one column at a time, there must be some expansion move prior to which the sum of the dimensions of the secure set is $n$ and after which the sum of the dimensions is $n+1$.  We may assume that the secure set was expanded from size $a \times b$ to size $(a+1)\times b$, where $a+b=n$; the case where we expand from $a \times b$ to $a \times (b+1)$ is similar.  As noted above, this expansion requires at least $\ceil{\frac{b(n-a) + (a+1)(n-b)}{2}}$ cops.  Taking into account that $a+b=n$ and hence $b=n-a$, we see that the number of cops needed is at least 
% \frac{(n-a)^2 + (a+1)a}{2} = \frac{n^2-2an+a^2+a^2+a}{2} = a^2 + n^2/2 - an + a/2
$a^2 - (n-\frac{1}{2})a + \frac{n^2}{2}$.  Treating $a$ as a real-valued variable, we find that the function is minimized when $a=\frac{n}{2} - \frac{1}{4}$; 
% d/da gives 2a - n + 1/2; this is 0 when 2a = n-1/2
at this value of $a$, the bound on the number of cops is $\frac{n^2 + n}{4} - \frac{1}{16}$.  
% (n/2-1/4)^2 - (n-1/2)(n/2-1/4) + n^2/2 = (1/4)(n-1/2)^2 - (1/2)(n-1/2)^2 + n^2/2 = -(1/4)(n^2-n+1/4) + n^2/2 = n^2/4 + n/4 - 1/16
Thus the number of cops used by any winning weakly monotone cop strategy must be at least $\ceil{\frac{n^2+n}{4}-\frac{1}{16}}$, which is equivalent to $\ceil{\frac{n^2+n}{4}}$ for integer values of $n$.  Thus $\search_0(H(2,n)) \ge \ceil{\frac{n^2+n}{4}}$, as claimed.
\end{proof}
\,
\end{section}

\begin{section}{Higher-Dimensional Hamming Graphs}\label{sec:high_dimension}

In this section, we turn our attention to higher-dimensional Hamming graphs.  The $d$-dimensional Hamming graph $H(d,n)$ has diameter $d$, so when $\ell \ge d$ we have $c_{\ell}(H(d,n)) = c(H(d,n))$; Neufeld and Nowakowski \cite{NN98} showed that $c(H,(d,n)) = d$, hence also $c_{\ell}(H(d,n)) = d$.  For $\ell \le d-1$, the situation is less clear.  

Theorem \ref{thm:seeing_plus_one} and the aforementioned result of Neufeld and Nowakowski together imply that for $\ell \ge 2$, we have
\[\lloc(H(d,n)) \le \lcap(H(d,n)) \le \max\{\lloc(H(d,n)), d+1\}.\]
In general, for $d \ge 2$ and $2 \le \ell < d$, we will see that $\lloc(H(d,n))$ is much larger than $d$; consequently, in such situations $\lcap(H(d,n)) = \lloc(H(d,n))$.  Thus, we will focus exclusively on the problem of seeing (but not capturing) a robber on $H(d,n)$.

We begin with some observations that, together with Theorem \ref{thm:clique_product_search}, will yield an easy upper bound on $\lloc(H(d,n))$; the remainder of the section is devoted to obtaining a more difficult lower bound on $\search_{d-1}(H(d,n))$.

\begin{proposition}\label{prop:Gnd_upper}
For all positive integers $n$, $d$, and $\ell$ with $n \ge 3$, we have $\search_{\ell+1}(H(d+1,n)) \le \search_{\ell}(H(d,n))$.
\end{proposition}
\begin{proof}
Let $k = \search_{\ell}(H(d,n))$ and consider the $(\ell+1)$-visibility game on $H(d+1,n)$ with $k$ cops.  To find the robber on $H(d+1,n)$, each cop can change the first $d$ coordinates of her position according to a strategy for seeing the robber in the $\ell$-visibility game on $H(d,n)$, while keeping the last coordinate of her position equal to 1.

To show that this suffices to locate the robber on $H(d+1, n)$, we claim that at all times, if vertex $(v_1, \dots, v_d)$ is clean in the imagined $\ell$-visibility game on $H(d,n)$, then for all $v_{d+1} \in \{1, \dots, n\}$, vertex $(v_1, \dots, v_d, v_{d+1})$ is clean in the $(\ell+1)$-visibility game on $H(d+1,n)$.  Suppose for the sake of contradiction that this is not the case, and consider the first point in time in which some vertex $x = (x_1, \dots, x_d)$ is clean in the $\ell$-visibility game, but $x' = (x_1, \dots, x_d, x_{d+1})$ is contaminated in the $(\ell+1)$-visibility game for some $d+1$.  

Note that if a cop sees $x$ in the imagined $\ell$-visibility game, then there are at least $d-\ell$ values of $i$ such that the $i$th coordinate of her position is $x_i$.  Consequently, that cop must also see $x'$ in the $(\ell+1)$-visibility game on $H(d+1,n)$.  Hence, the first point at which $x$ is clean in the $\ell$-visibility game, but $x'$ is contaminated in the $(\ell+1)$-visibility game must occur immediately after the recontamination phase of some round, say round $r$.  In particular, it must be the case that both $x$ and $x'$ were clean in their respective games immediately following the cop move in round $r$, but during the ensuing recontamination phase, $x'$ became contaminated and $x$ did not.  

It now follows that in the $(\ell+1)$-visibility game, contamination must have spread to $x'$ from some vertex $w' = (w_1, \dots, w_d, w_{d+1})$ where $w_i \not = v_i$ for exactly one $i$.  Thus, $w'$ was dirty prior to recontamination in round $r$; hence, by choice of $x$ and $x'$, it must be that $w = (w_1, \dots, w_d)$ was also dirty at that time.  Furthermore, because $x'$ becomes dirty after recontamination in round $r$, no cop could possibly see that vertex in the $(\ell+1)$-visibility game, and hence no cop sees $x$ in the $\ell$-visibility game.  Consequently, either $w = x$, or contamination should have spread from $w$ to $x$; either possibility contradicts the assumption that $x$ was clean after the recontamination phase of round $r$.

%To see this, first note that if a cop would see vertex $(v_1, \dots, v_{d})$ in the $\ell$-visibility game on $H(d,n)$, then there are at least $d-\ell$ values of $i$ such that the $i$th coordinate of her position is $v_i$.  Consequently, in the $(\ell+1)$-visibility game on $H(d+1,n)$, the cop sees all vertices of the form $(v_1, \dots, v_{d}, w)$.  Let $S$ denote the set of all such vertices.  
%Next, suppose that some vertex in $S$ later becomes recontaminated in the $(\ell+1)$-visibility game on $H(d+1,n)$; we claim that when this happens, $(v_1, \dots, v_d)$ must already be dirty in the $\ell$-visibility game on $H(d,n)$.  Fix $v \in S$ such that no vertex of $S$ becomes recontaminated earlier than $v$, and let $v = (v_1, \dots, v_d, v_{d+1})$.  By choice of $v$, the contamination cannot have spread to $v$ from any vertex in $S$.  Hence, the contamination must have spread to $v$ from some vertex $w = (w_1, \dots, w_{d+1})$ such that $v_i \not = w_i$ for exactly one $i \in \{1, \dots, d\}$.  Thus, $w$ was dirty immediately prior to $v$'s recontamination; hence, vertex $(w_1, \dots, w_d)$ must have been dirty in the $\ell$-visibility game on $H(d,n)$.  Once $v$ becomes recontaminated, vertex $(v_1, \dots, v_d)$ must also be dirty: it cannot currently be seen by a cop (as that cop would also see $v$ in the game on $H(d+1,n)$), and it is adjacent to the dirty vertex $(w_1, \dots, w_d)$.  This establishes the claim.

This establishes the claim that if some vertex $(v_1, \dots, v_d)$ is clean in the imagined $\ell$-visibility game on $H(d,n)$, then all vertices of the form $(v_1, \dots, v_d, v_{d+1})$ are clean in the $(\ell+1)$-visibility game on $H(d+1, n)$.  Since the cops can eventually see all of $H(d,n)$ in the $\ell$-visibility game, it follows that they can eventually see all of $H(d+1,n)$ in the $(\ell+1)$-visibility game. 
\end{proof}

We remark that the cop strategy in Proposition \ref{prop:Gnd_upper} does not take full advantage of the added visibility in the $(\ell+1)$-visibility game; as such, it is likely to be far from tight in most (if not all) cases.

\begin{proposition}\label{prop:Gnd_lower_visibility_upper_bound}
For positive integers $n$, $\ell$, and $d$ with $n \ge 3$, we have $\lloc(H(d,n)) \le n \cdot \search_{\ell+1}(H(d,n))$.
\end{proposition}
\begin{proof}
Let $k = \search_{\ell+1}(H(d,n))$; we give a strategy for $kn$ cops to locate a robber in the $\ell$-visibility game on $H(d,n)$.  The cops partition themselves into $k$ groups of $n$ cops apiece.  For all $i \in \{1, \dots, k\}$, the $i$th group of cops in the $\ell$-visibility game mirrors the movements of the $i$th cop in a winning strategy for the $(\ell+1)$-visibility game: when that cop occupies vertex $(v_1, v_2, \dots, v_d)$, the corresponding cops in the $\ell$-visibility game occupy vertices $(1, v_2, \dots, v_d), (2, v_2, \dots, v_d), \dots, (n, v_2, \dots, v_d)$.  This ensures that whenever the $i$th cop sees a vertex in the $(\ell+1)$-visibility game, some cop in the $i$th group sees that vertex in the $\ell$-visibility game.  As in the proof of Proposition \ref{prop:Gnd_upper}, it follows that at all points during the game, any vertex that is clean in the $(\ell+1)$-visibility game must also be clean in the $\ell$-visibility game; since the cops eventually clean all vertices in the $(\ell+1)$-visibility game, they must do so in the $\ell$-visibility game as well.
\end{proof}
 
Theorem \ref{thm:clique_product_search}, Proposition \ref{prop:Gnd_upper}, and Proposition \ref{prop:Gnd_lower_visibility_upper_bound} yield a general upper bound on $c'_{\ell}(H(d,n))$.

\begin{cor}\label{cor:Gnd_lower_visibility_upper_bound}
For fixed positive integers $d$ and $\ell$ with $d \ge 2$ and $\ell \le d-1$, we have $c'_{\ell}(H(d,n)) \le n^{d-\ell}\left(\frac{1}{3} + o(1)\right)$.
\end{cor}

% Setup: dimension d, visibility k.  Want to show that $c'_k(H(d,n)) \ge f(d, k) \cdot n^{d-k} for some function f.
% Idea: use isoperimetry.  But maybe this won't work so well, as vertex-isoperimetry in Hamming graphs actually seems pretty tough to get a handle on.
% Maybe we can get a loose bound?
% It would be nice to prove that e.g. if at least 1/4 of the vertices are dirty before recontamination, then at least 1/2 will be dirty afterward.
% Suppose at least 1/4 of vertices are dirty.  If there's some layer in which at least half of the vertices are dirty, then recontamination will force at least half of _all_ vertices to be dirty.  Otherwise...how much recontamination must there be within a layer?
% Maybe: there exists some c such that if c*|V(G)| vertices are dirty before recontamination, then 2c*|V(G)| are dirty afterwards.

We remark without proof that the cop strategy used in Theorem \ref{thm:dimension_2_visibility_0} to establish an upper bound on $\search_0(H(2,n))$ -- creating and gradually expanding a secure set while keeping the dimensions roughly equal -- can be extended to higher dimensions to yield the bound $c_0(H(d,n)) \le n^d \cdot \frac{1}{2}\left(\frac{d-1}{d}\right)^{d-1}$.  As such, the constant $1/3$ in Corollary \ref{cor:Gnd_lower_visibility_upper_bound} is not, in general, optimal.  However, we strongly suspect that the order of magnitude is correct, i.e. that $c'_{\ell}(H(d,n)) = \Theta(n^{d-\ell})$.

Next, we turn to the problem of establishing a lower bound on $c'_{d-1}(H(d,n))$.  We begin by introducing some terminology that will help us reason about $H(d,n)$.  

We will be particularly interested in ``lower-dimensional'' induced subgraphs of $H(d,n)$.  Consider an ordered $d$-tuple in which each coordinate belongs to the set $\{1, \dots, n\} \cup \{\ast\}$.  For such an ordered $d$-tuple $(v_1, \dots, v_d)$, we define the \textit{slice of $G$ corresponding to $(v_1, \dots, v_d)$} to be the set of vertices 
$$\{(x_1, \dots, x_d) \, \vert \, x_i = v_i \text{ for all $v_i$ such that } v_i \not = \ast\}.$$
For example, when $d = 3$, the slice corresponding to $(1,\ast,\ast)$ consists of all vertices whose first coordinate is 1 (and whose second and third coordinates can take on any value); the slice corresponding to $(\ast, 4, 2)$ consists of all vertices whose second coordinate is 4 and whose third coordinate is 2.

%We say that a given cop \textit{sees} a slice $S$ if that cop simultaneously sees every vertex of $S$.

Each $\ast$ in a $d$-tuple represents a coordinate in which vertices of the corresponding slice are free to take on any value -- in a sense, it represents a  ``degree of freedom'' for vertices in the slice.  We define the \textit{dimension} of the slice corresponding to $(v_1, \dots, v_d)$ to be the number of $\ast$ among $v_1, \dots, v_d$.  Thus, for example, the dimension of the slice corresponding to $(\ast, 4, 2)$ is 1.  When a slice has dimension $m$, we refer to it as an \textit{$m$-slice}.  Note that each $m$-slice induces a subgraph isomorphic to $H(m,n)$.  In particular, the vertex set of $H(d,n)$ is itself a $d$-slice.  We sometimes use the term ``$m$-slice'' to refer to the subgraph of $H(d,n)$ induced by an $m$-slice.

We say that two $m$-slices are \textit{parallel} if their intersection is empty.  Note that for $m \ge 1$, an $m$-slice contains $mn$ distinct $(m-1)$-slices, which can be grouped into $m$ sets of $n$ pairwise-parallel $(m-1)$-slices.  Given a $m$-slice $S$, we refer to a set of $n$ pairwise-parallel $(m-1)$-slices in $S$ as a \textit{parallel class} of $(m-1)$-slices in $S$.  In particular, if $S$ corresponds to the $d$-tuple $(v_1, \dots, v_d)$, then each parallel class corresponds to a choice of one coordinate of $(v_1, \dots, v_d)$ that was $\ast$, and each slice in the class corresponds to a choice of value from $\{1, \dots, n\}$ to replace the $\ast$ with.

To establish a lower bound on $c'_{d-1}(H(d,n))$, we will need to argue that it is ``difficult'' for cops to clean $H(d,n)$.  Toward that end, we will want to argue that if the number of cops is ``not too large'', then $H(d,n)$ will remain ``highly dirty''.  Thus motivated, we define a 0-slice (i.e. a single vertex) to be \textit{good} if it is dirty, and for $m \in \{1, \dots, d\}$, we define an $m$-slice to be \textit{good} if it has at least $(m-1)k+1$ good $(m-1)$-slices in every parallel class, where $k$ denotes the number of cops in the game.  Similarly, we define a 0-slice to be \textit{very good} if it is dirty, and for $m \in \{1, \dots, d\}$, we define an $m$-slice to be \textit{very good} if it has at least $mk+1$ very good $(m-1)$-slices in each parallel class.  

Before proceeding, we establish several properties of slices and how they interact with each other.

\begin{lemma}\label{lem:slice_intersection}
The slice of $H(d,n)$ corresponding to $(v_1, \dots, v_d)$ contains the slice corresponding to $(w_1, \dots, w_d)$ if and only if $v_i \in \{w_i, \ast\}$ for all $i$.
\end{lemma}
\begin{proof}
Let $S$ denote the slice corresponding to $(v_1, \dots, v_d)$, let $T$ denote the slice corresponding to $(w_1, \dots, w_d)$, and suppose that $S$ contains $T$.  It will be useful to note that by definition of $T$, the vertex $(x_1, \dots, x_d)$ belongs to $T$ if and only if in each coordinate, either $w_i = \ast$ or $w_i = x_i$.%; equivalently, $w_i \in \{x_i, \ast\}$ for all $i$.  

For the reverse direction of the lemma, suppose that $v_i \in \{w_i, \ast\}$ for all $i$.  It follows that $v_i \in \{w_i, \ast\} \subseteq \{x_i, \ast\}$, hence $(x_1, \dots, x_d)$ belongs to $S$ as well.  For the forward direction, suppose for the sake of contradiction that for some $i$ we have $v_i \not \in \{w_i, \ast\}$.  If $w_i \not = \ast$, then let $x_j = 1$ if $w_j = \ast$ and let $x_j = w_j$ otherwise; now $v_i \not \in \{w_i, \ast\}$, hence the vertex $(x_1, \dots, x_d)$ belongs to $T$ but not to $S$, contradicting the assumption that $S$ contains $T$.  Similarly, if $w_i = \ast$, then let $x_i = v_i + 1 (\textrm{mod }n)$, let $x_j = 1$ if $j \not = 1$ and $w_j = \ast$, and let $x_j = w_j$ otherwise; once again, the vertex $(x_1, \dots, x_d)$ belongs to $T$ but not to $S$.
\end{proof}

We say that a cop \textit{sees} a slice (or \textit{sees all of} a slice) if that cop simultaneously sees every vertex of that slice.

\begin{lemma}\label{lem:slice_properties}
Consider the $(d-1)$-visibility game on $H(d,n)$ with $k$ cops, where $n \ge 2k+1$.  Let $S$ represent an arbitrary $m$-slice of $H(d,n)$ for some $m \in \{0, \dots, d-1\}$.
\begin{enumerate}
\item [\textbf{(1)}] If $m \ge 1$, then at any given point in the game, a given cop either sees all of $S$ or sees exactly one $(m-1)$-slice in each parallel class of $S$.  
\item [\textbf{(2)}] If $T$ and $T'$ are distinct $(m+1)$-slices containing $S$, then $T \cap T' = S$.  Moreover, if no cop sees either $T$ or $T'$, then no cop sees $S$.
\item [\textbf{(3)}] Let $T$ be an $(m+1)$-slice that contains $S$.  If $T$ is good prior to the recontamination phase of some round $r$, and if no cop currently sees all of $S$, then $S$ will be very good prior to the cop move phase of round $r+1$. 
\item [\textbf{(4)}] If $S$ is very good prior to the cop move phase of some round and no cop sees all of $S$ after the cop move, then $S$ will be good after the cop move.
\end{enumerate}
\end{lemma}
\begin{proof}
By symmetry, it suffices to let $S$ be the $m$-slice corresponding to the $d$-tuple whose first $m$ coordinates are $\ast$ and whose last $d-m$ coordinates are 1.\\ 

\textbf{For (1)}, note that a cop $C$ at vertex $(x_1, \dots, x_d)$ sees all vertices $(y_1, \dots, y_d)$ where $x_i = y_i$ for at least one $i$ -- that is, $C$ sees the $(d-1)$-slices corresponding to $(x_1, \ast, \dots, \ast)$, $(\ast, x_2, \ast, \dots, \ast), \dots, (\ast, \dots, \ast, x_d)$.  Thus, by Lemma \ref{lem:slice_intersection}, $C$ sees all of $S$ provided that $x_i = 1$ for some $i \in \{m+1, \dots, d\}$; otherwise, $C$ sees the $(m-1)$-slices corresponding to $(x_1, \ast, \dots, \ast, 1, \dots, 1), (\ast, x_2, \dots, \ast, 1, \dots, 1), \dots, (\ast, \ast, \dots, x_m, 1, \dots, 1)$.  Each of these $(m-1)$-slices corresponds to a choice of value for one of the $m$ coordinates in which $(y_1, \dots, y_d)$ was $\ast$; in other words, it corresponds to an $(m-1)$-slice in one of the $m$ parallel classes in $S$.\\

\textbf{For (2)}, first note that by Lemma \ref{lem:slice_intersection}, both $T$ and $T'$ must correspond to $d$-tuples in which the first $m$ coordinates are $\ast$ and every other coordinate is either $\ast$ or 1.  Since both $T$ and $T'$ are $(m+1)$-slices, in each case exactly one of the last $d-m$ coordinates must be $\ast$ and the rest must be 1.  Hence, we may suppose by symmetry that $T$ corresponds to the $d$-tuple whose first $m+1$ coordinates are $\ast$ and whose last $d-m-1$ coordinates are $1$, while $T'$ corresponds to the $d$-tuple in which the first $m$ coordinates are $\ast$, coordinate $m+2$ is $\ast$, and the remaining coordinates are all 1.  It now follows from Lemma \ref{lem:slice_intersection} that $T \cap T'$ corresponds to the $d$-tuple in which the first $m$ coordinates are $\ast$ and the remaining coordinates are all 1, i.e. $T \cap T' = S$.   

Now suppose that no cop sees either $T$ or $T'$.  A cop that does not see $T$ must not have any 1s among its last $d-m-1$ coordinates; a cop that does not see $T'$ must not have any 1s in coordinate $m+1$ or in its last $d-m-2$ coordinates.  Hence, a cop that sees neither $T$ nor $T'$ must not have any 1s among its last $d-m$ coordinates and hence cannot see $S$.\\

\textbf{For (3)}, we use induction on $m$.  For the base case, suppose $m=0$.  Note that $S$ is a single vertex and $T$ is a 1-slice, i.e. $T$ is a clique.  If $T$ is good prior to the recontamination phase of round $r$, then it must have at least one good 0-slice at that point in time.  During the recontamination phase, contamination will spread from this good 0-slice to all other 0-slices in $T$.  In particular, $S$ will be dirty after recontamination; if no cop sees it, then it will remain dirty (and hence very good) prior to the cop move in round $r+1$.  
%In particular, $S$ will be completely dirty after the recontamination phase of round $r$, save for any vertices seen by cops.  No cop sees all of $S$, and $S$ has only one parallel class, so each of the $k$ cops sees exactly one $0$-slice (that is, one vertex) in $S$.  Thus, prior to the cop move in round $r+1$, there are $n-k$ dirty vertices in $S$; since $n-k \ge k+1$, it follows that $S$ is very good.

Suppose then that $m \ge 1$.  Let $\mathcal{C}_S$ be an arbitrary parallel class of $(m-1)$-slices in $S$; to show that $S$ will be very good prior to the cop move phase of round $r+1$, we aim to show that $\mathcal{C}_S$ will contain at least $mk+1$ very good $(m-1)$-slices at that time.  By symmetry, we may suppose that $T$ corresponds to the $d$-tuple whose first $m+1$ coordinates are $\ast$ and whose remaining coordinates are 1, and that $\mathcal{C}_S$ is the parallel class of $(m-1)$-slices in $S$ corresponding to a choice of value for the $\ast$ in coordinate $m$. Let $\mathcal{C}_T$ be the parallel class of $m$-slices in $T$ corresponding to a choice of value for the $\ast$ in coordinate $m$, and note that for every $T' \in \mathcal{C}_T$ we have $S \cap T' \in \mathcal{C}_S$.

Since $T$ was good prior to the recontamination phase of round $r$, at that time it had at least $mk+1$ good $m$-slices in each parallel class; in particular, $\mathcal{C}_T$ had at least $mk+1$ good $m$-slices, say $T_1, \dots, T_{mk+1}$.  For $1 \le i \le mk+1$, let $S_i = S \cap T_i$.  Note that the $S_i$ all belong to $\mathcal{C}_S$ and that each $S_i$ is contained within the good $m$-slice $T_i$.   Moreover, since the $T_i$ were good prior to recontamination in round $r$, no cop can have seen any $T_i$ during round $r$ -- and hence no cop sees any $T_i$ prior to the cop move phase of round $r+1$.  By assumption, no cop sees $S$ prior to the cop move phase of round $r+1$; thus, it follows from property (2) that no cop sees any $S_i$ prior to the cop move phase of round $r+1$.  By induction, since the $S_i$ were contained within good $m$-slices prior to the recontamination phase of round $r$, and no cop sees any $S_i$ prior to the cop move phase of round $r+1$, the $S_i$ will be very good prior to the cop move phase of round $r+1$; hence at that point in time $\mathcal{C}_S$ will contain at least $mk+1$ very good $(m-1)$-slices, as claimed.\\

\textbf{For (4)}, we again use induction on $m$.  When $m=0$, the claim is trivial: $S$ is a single vertex; if $S$ is very good prior to a cop move, then it is dirty and will remain that way after the cop move, provided no cop moves to see it.

Suppose then that $m \ge 1$.  Fix a parallel class $\mathcal{C}_S$ of $(m-1)$-slices in $S$; we must show that $\mathcal{C}_S$ has at least $(m-1)k+1$ good slices after the cop move, provided that no cop sees all of $S$.  Since $S$ is very good, $\mathcal{C}_S$ must have at least $mk+1$ very good $(m-1)$-slices, say $S_1, \dots, S_{mk+1}$.  By the induction hypothesis, each $S_i$ will remain good unless some cop sees it after the cop move.  However, by property (1), since no cop sees all of $S$ after the cop move, each cop will only see one $(m-1)$-slice from each parallel class in $S$; in particular, each cop will see at most one $S_i$.  Hence at most $k$ of the $S_i$ can fail to be good after the cop move, so at least $(m-1)k+1$ of them will still be good, as claimed. 
\end{proof}

We are now ready to establish a general lower bound on $c'_{d-1}(H(d,n))$.

\begin{theorem}\label{thm:Gnd_lower}
For positive integers $n$ and $d$ with $n \ge 3$, we have $c'_{d-1}(H(d,n)) \ge \ceil{n/(d+1)}$.  
\end{theorem}
\begin{proof}
Let $G = H(d,n)$ and let $k = \floor{(n-1)/(d+1)}$; we aim to show that $k$ cops cannot capture a robber on $G$.  Toward this end, we claim that for all $m \in \{0, \dots, d-1\}$ and all positive integers $r$, immediately prior to the recontamination phase of round $r$, every $m$-slice that was not seen during rounds $r, r-1, \dots, r-d+m+1$ is good; we additionally claim that $G$ itself (viewed as a $d$-slice) is good prior to the recontamination phase of every round.  Suppose otherwise, and let $r$ denote the first round in which, immediately preceding the recontamination phase, either $G$ itself was not good, or some $m$-slice $S$ was not good despite not being seen in rounds $r, r-1, \dots, r-d+m+1$.
%; moreover, choose $S$ so as to minimize the value of $m$.

Suppose first that for some $m \in \{0, \dots, d-1\}$, some $m$-slice $S$ was not good despite not being seen in rounds $r, r-1, \dots, r-d+m+1$.  Let $T$ be any $(m+1)$-slice containing $S$. Since $S$ was not seen during rounds $r, r-1, \dots, r-d+m+1$, neither was $T$.  Consequently, by choice of $r$, $T$ must have been good prior to the recontamination phase of round $r-1$.  Since $T$ was good before recontamination in round $r-1$, by Lemma \ref{lem:slice_properties} (3), $S$ was very good immediately prior to the cop move phase of round $r$; by \ref{lem:slice_properties} (4) and the assumption that $S$ is not seen in round $r$, it will remain good prior to the recontamination phase of round $r$, contradicting the choice of $S$.  (Note that $k = \floor{(n-1)/(d+1)}$ implies $n \ge (d+1)k+1 \ge 2k+1$, so the hypotheses of Lemma \ref{lem:slice_properties} are satisfied.)

Suppose instead that, immediately prior to the recontamination phase of round $r$, for all $m \in \{0, \dots, d-1\}$, every $m$-slice that was not seen in rounds $r, r-1, \dots, r-d+m+1$ was good, but $G$ itself was not good.  Let $\mathcal{C}$ be an arbitrary parallel class of $(d-1)$-slices in $G$.  By choice of $r$, every $(d-1)$-slice in $\mathcal{C}$ that was not seen during round $r$ was good prior to the recontamination phase.  Each cop sees at most one $(d-1)$-slice in $\mathcal{C}$ at the beginning of round $r$ and can move to see one more.  Thus, at most $2k$ of the $(d-1)$-slices in $\mathcal{C}$ are seen during round $r$; hence, at least $n-2k$ of them are still good prior to the recontamination phase.  Since $k = \floor{(n-1)/(d+1)}$, we have $n > k(d+1)$, hence $n-2k > k(d+1)-2k = (d-1)k$, so $\mathcal{C}$ has at least $(d-1)k+1$ good $(d-1)$-slices; it follows that $G$ itself is good.

Since $G$ is good prior to the recontamination phase of every round, it must always contain dirty vertices, so the cops are unable to see the robber.
\end{proof}

It follows from Theorem \ref{thm:Gnd_lower}, Theorem \ref{thm:clique_product_search}, and Proposition \ref{prop:Gnd_upper} that for all positive integers $n$ and $d$ with $n \ge 3$ and $d \ge 2$, we have
\[\ceil{\frac{n}{d+1}} \le c'_{d-1}(H(d,n)) \le \ceil{\frac{n+1}{3}}.\]
As noted after Proposition \ref{prop:Gnd_upper}, the argument used to establish the upper bound therein is somewhat wasteful, in that the argument does not take full advantage of the cops' visibility.  Consequently, we believe that the lower bound above should be closer to the true value of $c'_{d-1}(H(d,n))$; indeed, we believe that the lower bound is asymptotically tight.

\begin{conj}\label{conj:Gnd}
For any fixed positive integer $d$, we have $c'_{d-1}(H(d,n)) = (1+o(1))\frac{n}{d+1}$.
\end{conj}
\,
\end{section}

\begin{section}{Consequences, Conclusions, and Questions}\label{sec:conclusions}

In this section, we use the results of Sections \ref{sec:dimension_2} and \ref{sec:high_dimension} to resolve two of the open problems posed by Clarke et al. in \cite{CCDDFM20} and to partially answer a third.  We begin with the following question:

\begin{problem}[\cite{CCDDFM20}, Problem 5.3]
For $\ell \ge 1$, can $\lcap(G) - \lloc(G)$ be arbitrarily large?
\end{problem}

We answer this question in the affirmative.  We remark that this question was answered independently and contemporaneously by Ba\v{s}i\'c, Davies, D\v{z}uklevski, Gvozdi\'c, and Mogge \cite{Dav25}.

\begin{theorem}
$\lcap(G) - \lloc(G)$ can be arbitrarily large for every $\ell \ge 1$.
\end{theorem}
\begin{proof}
When $\ell \ge 2$, it suffices to consider incidence graphs of projective planes, i.e. bipartite graphs with one vertex for each point and one vertex for each line in a given projective plane, with a ``point vertex'' adjacent to a ``line vertex'' when the corresponding point belongs to the corresponding line.  It was shown in \cite{Pra10} that when $G$ is an incidence graph of a projective plane of order $q$, we have $c(G) = q+1$; consequently, $\lcap(G) \ge c(G) = q+1$.  On the other hand, in the $\ell$-visibility game with $\ell \ge 2$, a single cop who starts on any ``point vertex'' and moves to  any ``line vertex'' can, from those two vantage points, see all vertices of $G$; hence $\lloc(G) = 1$.  Thus we have $\lcap(G) - \lloc(G) \ge q$; since projective planes of arbitrarily large order exist, $\lcap(G) - \lloc(G)$ can be arbitrarily large.

For the case $\ell = 1$, consider $G = H(2,n)$.  Theorem \ref{thm:clique_product_search} shows that $c'_1(G) = \ceil{\frac{n+1}{3}}$, and Theorem \ref{thm:clique_product_capture} shows that $c_1(G) = \ceil{\frac{n+1}{2}}$; thus $c_1(G) - c'_1(G) = \ceil{\frac{n+1}{2}} - \ceil{\frac{n+1}{3}}$, which grows unboundedly with $n$.
\end{proof}

We next turn to Problem 5.5 from \cite{CCDDFM20}.

\begin{problem}[\cite{CCDDFM20}, Problem 5.5]
What is the closure of the set of possible values of $\dfrac{c_{\ell}(G)}{c_{\ell+1}(G)}$? 
\end{problem}

For any graph $G$ and nonnegative integer $\ell$ we clearly have $c_{\ell+1}(G) \le c_{\ell}(G)$, since the cops are at least as strong in the $(\ell+1)$-visibility game as in the $\ell$-visibility game.  Thus, the ratio $c_{\ell}(G) / c_{\ell+1}(G)$ is always at least 1.  We will show that in fact this ratio can take on any rational value that is at least 1, and hence that the closure of the set of possible values is the interval $[1, \infty)$.  

We begin with two helpful lemmas regarding the behavior of $\lcap(H(d,n))$ and $\lloc(H(d,n))$.

\begin{lemma}\label{lem:diff_at_most_one}
For any positive integers $n \ge 3$ and $d \ge 3$, we have $c'_{d-1}(H(d,n)) \le c'_{d-1}(H(d,n+1)) \le c'_{d-1}(H(d,n))+1$ and $c_{d-1}(H(d,n)) \le c_{d-1}(H(d,n+1)) \le c_{d-1}(H(d,n))+1$.
\end{lemma}
\begin{proof}
It is straightforward to verify that $H(d,n)$ is a retract of $H(d,n+1)$  under the map $\phi : V(H(d,n+1)) \rightarrow V(H(d,n))$ defined by $\phi((x_1,\dots,x_d)) = (\min\{x_1,n\}, \dots, \min\{x_d,n\})$.  Thus, by Theorem \ref{thm:monotone_retract}, we have $c'_{d-1}(H(d,n)) \le c'_{d-1}(H(d,n+1))$ and $c_{d-1}(H(d,n)) \le c_{d-1}(H(d,n+1))$.  

To see that $c'_{d-1}(H(d,n+1)) \le c'_{d-1}(H(d,n))+1$, note that $c'_{d-1}(H(d,n))+1$ cops can locate a robber in the $(d-1)$-visibility game on $H(d,n+1)$ by positioning one cop at vertex $(n+1, \dots, n+1)$ and having the other cops execute a winning strategy for the $(d-1)$-visibility game on the copy of $H(d,n)$ induced by $\{(x_1, \dots, x_d) \in V(H(d,n+1)) \, : \, 1 \le x_i \le n \text{ for all } i\}$.  The cop on $(n+1, \dots, n+1)$ always sees all vertices with one or more coordinates equal to $n+1$, so those vertices will never be dirty; meanwhile, the remaining $\search_{d-1}(H(d,n))$ cops will clean the remainder of the graph.

Finally, we will argue that $c_{d-1}(H(d,n+1)) \le c_{d-1}(H(d,n))+1$.  By Theorem \ref{thm:seeing_plus_one} and the fact that $c(H(d,n+1)) = d$, we have, 
\begin{align*}
c_{d-1}(H(d,n+1)) &\le \max\{c'_{d-1}(H(d,n+1)), c(H(d,n+1))+1\}\\
                  &= \max\{c'_{d-1}(H(d,n+1)), d+1\}\\
                  &\le \max\{c'_{d-1}(H(d,n))+1, d+1\}\\
                  &\le \max\{c_{d-1}(H(d,n))+1, d+1\}.
\end{align*}
Since $c_{d-1}(H(d,n)) \ge c(H(d,n)) = d$, we have $\max\{c_{d-1}(H(d,n))+1, d+1\} = c_{d-1}(H(d,n))+1$, hence $c_{d-1}(H(d,n+1)) \le c_{d-1}(H(d,n))+1$ as claimed.
\end{proof}

\begin{lemma}\label{lem:product_of_cliques_any_value}
For any positive integers $d$ and $m$ with $d \ge 3$ and $m \ge d$, there exists some $n$ such that $c_{d-1}(H(d,n)) = m$.  
\end{lemma}
\begin{proof}
We first claim that $c_{d-1}(H(d,d)) = d$.  Neufeld and Nowakowski \cite{NN98} showed that $c(H(d,n)) = d$ whenever $n \ge 3$; consequently, $c_{d-1}(H(d,d)) \ge c(H(d,d)) = d$.  For the upper bound, we explain how $d$ cops can capture a robber on $H(d,d)$.  Label the cops $C_1, \dots, C_d$.  Initially, cop $C_i$ begins on vertex $(i,1,1,\dots,1)$.  Note that no matter where the robber begins, his position must agree in the first coordinate with one of the cops' positions, and hence the cops are guaranteed to see him.  Henceforth, on each turn, the cops move as follows: each cop $C_i$ chooses the first coordinate \textit{other than coordinate $i$} in which her position disagrees with the robber and changes it to agree with the robber; if she already agrees with the robber in all coordinates except coordinate $i$, then she can and does capture the robber.  Note that in the course of each round, the number of coordinates in which each cop agrees with the robber cannot decrease; moreover, if on the robber turn the robber changes coordinate $i$ of his current position, then in fact cop $C_i$ increases the number of coordinates in which she agrees with the robber.  Hence, the robber can only move in any given coordinate at most $d-1$ times before being captured, so after at most $d(d-1)+1$ rounds, the cops will win.

Thus, $c_{d-1}(H(d,n)) = d$ when $n=d$.  By Lemma \ref{lem:diff_at_most_one}, increasing $n$ by 1 can increase $c_{d-1}(H(d,n))$ by at most 1; by Theorem \ref{thm:Gnd_lower}, $c_{d-1}(H(d,n))$ grows arbitrarily large as $n$ increases.  Thus, as $n$ increases, $c_{d-1}(H(d,n))$ must take on every integer value from $d$ onward; in particular, there must be some value of $n$ such that $c_{d-1}(H(d,n)) = m$, as claimed.
\end{proof}

We are now ready to answer Problem 5.5 from \cite{CCDDFM20}.

\begin{theorem}
The closure of the set $\displaystyle \left \{\frac{c_{\ell}(G)}{c_{\ell+1}(G)} \, : \, \ell \in \mathbb{Z}, G \text{ a graph}\right \}$ is $[1, \infty)$.
\end{theorem}
\begin{proof}
We claim that for all positive integers $m$ and $d$ with $m \ge d$, there exists a graph $G$ and positive integer $\ell$ such that $c_{\ell}(G) = m$ and $c_{\ell+1}(G) = d$; the theorem then follows immediately.  By Lemma \ref{lem:product_of_cliques_any_value}, there exists some $n$ such that $c_{d-1}(H(d,n)) = m$.  In addition, since $H(d,n)$ has diameter $d$, we must have $c_d(H(d,n)) = c(H(d,n))$, which was shown to equal $d$ by Neufeld and Nowakowski (\cite{NN98}, Theorem 2.9).
\end{proof}

Finally, we turn our attention to Problem 5.1 from \cite{CCDDFM20}.

\begin{problem}[\cite{CCDDFM20}, Problem 5.1]
Does there exist a graph $G$ for which
\[c_0(G) > c_1(G) > \dots > c_{\rad(G)}(G) > c(G)?\]
\end{problem}

While we are unable to resolve the question in whole, we give a construction that \textit{nearly} resolves the question in the affirmative.  We begin with a lemma.  As mentioned earlier, Neufeld and Nowakowski \cite{NN98} showed that for all positive integers $d$ and $n$, we have $c(H(d,n)) = d$.  We will need a slightly stronger form of this result; in particular, we will need to guarantee that the cops can not only capture the robber, they can do so while preventing him from ever reaching a specific vertex.

\begin{lemma}\label{lem:cop_number_hamming_protect}
If $G = H(d,n)$ for some positive integers $d$ and $n$, then $d$ cops have a strategy to capture a robber on $G$ wherein the robber can never reach vertex $(1,\dots, 1)$ without being captured on the subsequent cop turn.
\end{lemma}
\begin{proof}
Consider the game on $G$ with $d$ cops, and label the cops $C_1, \dots, C_d$. 
Initially, all $d$ cops begin the game on vertex $(1, \dots, 1)$.  Henceforth, the cops play as follows.  For all $i \in \{1, \dots, d\}$, in each round of the game, cop $C_i$ moves as follows.  Let $L_i$ denote the list $i+1, i+2, \dots, r, 1, \dots, i$.  Let $j$ denote the earliest entry in this list such that $C_i$'s current position disagrees with the robber's position in coordinate $j$; $C_i$ moves so as to agree with the robber in coordinate $j$.

We claim that this strategy has the desired properties.  We first argue that the cops eventually capture the robber.  To prove this, we will argue that in each round of the game, at least one cop ``makes progress'' toward capturing the robber.  To formalize this notion, for all $i \in \{1, \dots, d\}$, at all points during the game let $t_i$ denote the maximum $t$ such that cop $C_i$'s position agrees with the robber's position in the coordinates corresponding to the first $t$ entries of $L_i$.  Note that $t_i = d$ if and only if cop $C_i$ has captured the robber.  We claim that in the course of each robber turn and subsequent cop turn, no $t_i$ can ever decrease and, moreover, at least one $t_i$ must increase.  

To see this, suppose that the robber has not yet been captured, and consider one robber turn together with the ensuing cop turn.  Suppose first that the robber doesn't move on his turn.  Each cop $C_i$ moves to agree with the robber in the coordinate corresponding to entry $t_i+1$ of $L_i$; thus, each $t_i$ increases by at least 1.  Suppose instead that the robber moves by changing coordinate $k$ of his current position.  Every cop has the opportunity to ``catch up'' with the robber in this coordinate, if her strategy should so dictate; hence, no $t_i$ will decrease over the course of both turns.  Furthermore, since the robber had not been captured prior to his move, no cop agreed with the robber in all $d$ coordinates; in particular, we had $t_k \le d-1$.  Since coordinate $k$ is the last entry in $L_k$, it follows that $t_k$ did not decrease after the robber's move; moreover, cop $C_k$'s strategy will cause $t_k$ to increase after the ensuing cop turn.  Since the $t_i$ are weakly increasing, eventually we must have $t_i = d$ for some $i$, which implies that cop $C_i$ has captured the robber.  Hence, this is a winning strategy for the cops.

Finally, we claim that the strategy prevents the robber from ever reaching vertex $(1,\dots, 1)$ without being captured on the subsequent cop turn.  To see this, fix $\ell \in \{1, \dots, d\}$, and suppose that coordinate $\ell$ of the robber's position is 1 immediately after some cop turn.  If $t_d \ge \ell$, then cop $C_d$ must agree with the robber in coordinate $\ell$, so coordinate $\ell$ of $C_d$'s position is 1.  If instead $t_d < \ell$, then $C_d$ must never have moved in coordinate $\ell$; since coordinate $\ell$ of $C_d$'s position was 1 at the beginning of the game, it must still be 1 now.  It follows that $C_d$ agrees with the robber in every coordinate in which the robber's position is 1; consequently, if the robber ever moves to $(1, \dots, 1)$, then $C_d$ will capture him on the subsequent cop turn, as claimed.
\end{proof}

We are now ready to give a partial answer to Problem 5.1 from \cite{CCDDFM20}.  

\begin{theorem}\label{thm:cop_number_descending}
For every positive integer $r$, there exists a graph $G$ with $\rad(G) = r$ and 
\[c_0(G) > c_1(G) > \dots > c_r(G) = c(G) = r.\]
Moreover, if there exists a graph $G_r$ with $\rad(G_r) = r$, $c_r(G_r) = r+1$, and $c(G_r) = r$, then there exists a graph $G'$ with $rad(G') = r$ and 
\[c_0(G') > c_1(G') > \dots > c_r(G') > c(G') = r.\]
\end{theorem}
\begin{proof}
We begin by constructing a graph $G$ with $\rad(G) = r$ and 
\[c_0(G) > c_1(G) > \dots > c_r(G) = c(G) = r.\]
Loosely, we will do this by combining $r$ suitably-chosen Hamming graphs.

We first let $M_{r-1} = r+1$.  Let $n_{r-1}$ be the minimum $n$ such that $c_{r-1}(H(r,n)) = M_{r-1}$, the existence of which is guaranteed by Lemma \ref{lem:product_of_cliques_any_value}.  Moreover, due to the minimality of $n_{r-1}$, $M_{r-1}$ cops can win the $(r-1)$-visibility game on $H(r,n_{r-1})$ with one cop perpetually occupying vertex $(1, \dots, 1)$.  (This follows from essentially the same argument used to prove the upper bound in Lemma \ref{lem:diff_at_most_one}.)  Let $G_{r-1} = H(r,n_{r-1})$.

Next, let $M_{r-2} = c_{r-2}(G_{r-1})+1$ (and note that $M_{r-2} > c_{r-1}(G_{r-1}) = M_{r-1}$).  Once again, there exists some $n_{r-2}$ such that $c_{r-2}(H(r-1,n_{r-2})) = M_{r-2}$ and, moreover, $M_{r-2}$ cops can win the $(r-2)$-visibility game on $H(r-1,n_{r-2})$ with one cop perpetually occupying vertex $(1, \dots, 1)$; let $G_{r-2} = H(r-1,n_{r-2})$.

We continue in this manner.  For each $i$ from $r-3$ down to 0, we do the following:
\begin{itemize}
\item Let $M_i = \max\{c_i(G_{i+1}), \dots, c_i(G_{r-1})\}+1$;
\item Choose $n_i$ such that $c_{i}(H(i+1,n_i)) = M_i$ and, moreover, $M_i$ cops can win the $i$-visibility game on $H(i+1,n_i)$ with one cop perpetually occupying vertex $(1, \dots, 1)$; and
\item Let $G_i = H(i+1,n_i)$.
\end{itemize}
Note that the definition of $M_i$ immediately implies that $M_0 > M_1 > \dots > M_{r-1} > r$.

Finally, we construct the graph $G$ by identifying the ``all-ones'' vertices from $G_0, \dots, G_{r-1}$.  Let $v_0$ denote the identified vertex.  Since each $G_i$ has radius $i+1$ with the ``all-ones'' vertex as a central vertex, it is clear that $\rad(G) = r$.  We claim that $c_i(G) = M_i$ for all $i \in \{0, \dots, r-1\}$ and that $c_r(G) = c(G) = r$, from which the first part of the theorem will follow.

We first argue that $c_r(G) = c(G) = r$.  Since $G_{r-1}$ is a retract of $G$ (under the retraction that maps all vertices outside of $G_{r-1}$ to vertex $v_0$), we have $c_r(G) \ge c(G) \ge c(G_{r-1}) = c(H(r,n_{r-1})) = r$. 
To show that $c_r(G) \le r$, we give a strategy for $r$ cops to capture a robber in the $r$-visibility game on $G$; since $c(G) \le c_r(G)$, it will also follow that $c(G) \le r$.  All $r$ cops begin on $v_0$.  Because $v_0$ is a central vertex in $G$ and $\rad(G) = r$, the cops will be able to see the robber on their first turn.

If the robber initially occupies a vertex in $G_i$ for some $i \le r-2$, one cop remains on $v_0$, while the other $r-1$ cops follow a strategy to capture the robber in the $r$-visibility game on $G_i$.  The cop on $v_0$ ensures that the cops can see the robber at all times and, moreover, that the robber cannot move to $G_j$ for any $j \not = i$.  Since $G_i$ is a Hamming graph of dimension $i+1$, we have $c(G_i) = i+1 \le r-1$, so the remaining $r-1$ cops can indeed capture the robber on $G_i$.  Suppose instead that the robber initially occupies a vertex in $G_{r-1}$.  
%\comment{Can't we just use Lemma 5.5 here?} 
This time, we have $c(G_{r-1}) = c(H(r,n_{r-1})) = r$, so we do not have a ``free'' cop to position on $v_0$; however, by Lemma \ref{lem:cop_number_hamming_protect}, $r$ cops can capture a robber on $G_{r-1}$ without allowing him to access $v_0$ (thereby preventing him from escaping to any other $G_i$).
%Thus, all $r$ cops follow a strategy to capture the robber in the perfect-information game on $G_{r-1}$.  Since $G_{r-1}$ has radius $r$, the cops will always be able to see the robber unless he moves to $G_i$ for some $i \le r-2$.  If at any point the cops cannot see the robber, then they play in $G_i$ as if the robber occupies $(1, \dots, 1)$; note that the robber cannot reach a vertex of any other $G_i$ without first passing through $(1,\dots, 1)$ and thereby revealing himself to the cops.  Because the cops follow a winning strategy for the game on $G_i$, eventually either some cop actually captures the robber, or some cop reaches vertex $v_0$ while the robber inhabits $G_i$ for some $i \le r-2$.  In the former case, the game is over; in the latter case, the cops play as in the preceding paragraph, with the cop on $v_0$ remaining there for the rest of the game.  In either case, the cops win.

Now consider the $\ell$-visibility game with $\ell \in \{0, \dots, r-1\}$; we aim to show that $c_{\ell}(G) = M_{\ell}$.  By construction, we have $c_{\ell}(G_{\ell}) = M_{\ell}$; since $G_{\ell}$ is a retract of $G$, we have $c_{\ell}(G) \ge c_{\ell}(G_{\ell}) = M_{\ell}$.  To establish the matching upper bound, we give a strategy for $M_{\ell}$ cops to capture a robber on $G$.  One cop will occupy $v_0$ for the duration of the game; as before, this ensures that the robber can never move from $G_i$ to $G_j$ for $i \not = j$.  

If the robber initially occupies a vertex of $G_i$ for $i < \ell$, then the cop on $v_0$ will be able to see him at all times, and -- since $c(G_i) = i+1 \le r < M_{\ell}$ -- the remaining $M_{\ell}-1$ cops will be able to capture the robber.  If instead the robber occupies a vertex of $G_i$ for $i > \ell$, then the remaining $M_{\ell}-1$ cops iteratively search through $G_{\ell+1}, \dots, G_{r-1}$ until the robber is found; by construction $M_{\ell}-1 = \max\{c_{\ell}(G_{\ell+1}), \dots, c_{\ell}(G_{r-1})\}$, so $M_{\ell}-1$ cops suffice for this task.  Finally, suppose the robber initially occupies a vertex of $G_{\ell}$.  By construction, $M_{\ell}$ cops can capture a robber on $G_{\ell}$ even with one cop permanently occupying $v_0$; thus, once again, the cops can capture the robber.  It follows that the graph $G$ has the desired properties.

For the second part of the theorem, let $G_r$ be a graph with radius $r$ such that $c_r(G_r) = r+1$ and $c(G_r) = r$.  We show how to construct a graph $G'$ such that
\[c_0(G') > c_1(G') > \dots > c_r(G') > c(G') = r.\]
Our construction is nearly the same as that for $G$, with the following changes:
\begin{itemize}
\item Instead of taking $M_{r-1} = r+1$, we take $M_{r-1} = c_{r-1}(G_r)+1$;
\item Similarly, for $i < r-1$, instead of taking $M_i = \max\{c_i(G_{i+1}), \dots, c_i(G_{r-1})\}+1$, we take $M_i = \max\{c_i(G_{i+1}), \dots, c_i(G_{r-1}), c_i(G_r)\}+1$;
\item We form $G'$ by identifying the ``all-ones'' vertices from $G_0, G_1, \dots, G_{r-1}$, along with any central vertex from $G_r$.
\end{itemize}

It is clear from the construction that $\rad(G') = r$.  To see that $c(G') = r$, first note that since $G_r$ is a retract of $G'$, we have $c(G') \ge c(G_r) = r$.  For the matching upper bound, we explain how $r$ cops can capture a robber on $G'$.  The cops all begin on vertex $v_0$.  As before, if the robber begins in $G_i$ for some $i \in \{1, \dots, r-2\}$, then one cop  remains on $v_0$ -- thereby preventing the robber from fleeing $G_i$ -- while the remaining $r-1$ follow a winning strategy on $G_i$.  Similarly, if the robber begins in $G_{r-1}$, then Lemma \ref{lem:cop_number_hamming_protect} shows that the cops can capture the robber without letting him safely reach $v_0$ at any point.  Finally, suppose the robber begins in $G_r$.  The cops follow a winning strategy for the game on $G_r$; if the robber ever reaches $v_0$ and flees to some other $G_i$, then the cops continue playing on $G_r$ as if the robber were still on $v_0$.  Eventually, either the cops capture the robber, or some cop reaches $v_0$ while the robber inhabits some other $G_i$.  In the former case, the cops win; in the latter case, the cops can now capture the robber in $G_i$ as argued above.

Thus $c(G') = r$; arguments similar to those used above can be used to show that $c_{\ell}(G) = M_{\ell}$ for all $\ell \in \{0, \dots, r\}$.
\end{proof}

To conclude the paper, we leave the reader with a few open questions suggested by our work.
\begin{itemize}
\item Is it true that $c'_{\ell}(H(d,n)) = \Theta(n^{d-\ell})$?  Corollary \ref{cor:Gnd_lower_visibility_upper_bound} yields the needed asymptotic upper bound.  It seems likely to the authors that the lower bound could be established using a standard isoperimetric argument -- that is, arguing that with $o(n^{d-\ell})$ cops, when half (or so) of the vertices of $H(d,n)$ are dirty, vertices become recontaminated faster than the cops can clean them.  However, the details seem potentially messy.\\

\item Is it true that $c'_{d-1}(H(d,n)) = n/(d+1) + O(1)$?  Observation \ref{obs:simple} and Theorem \ref{thm:clique_product_search} show that this is the case when $d=1$ and $d=2$, and Theorem \ref{thm:Gnd_lower} provides the lower bound for general $d$, but a general upper bound eludes us.\\

\item Is there a graph $G$ with radius $r$ such that $c_r(G) = r+1$ and $c(G) = r$?  Together with Theorem 5.6, this would yield a full resolution of Problem 5.1 from \cite{CCDDFM20}.
\end{itemize}

\end{section}

\end{document}